\documentclass[amscd,11pt]{amsart}
\usepackage{authblk}
\usepackage{blindtext,enumerate}

\usepackage{graphicx,amssymb,amsmath,mathrsfs}
\usepackage{moreverb,bigfoot,footnote}
\usepackage[utf8]{inputenc}
\usepackage{epstopdf}
\usepackage{url,hyperref}	
\usepackage[boxed,linesnumbered]{algorithm2e}	
\usepackage{bm}
\usepackage[numbers]{natbib}
\usepackage{color}
\usepackage{listings}
\usepackage{array}

\makesavenoteenv{algorithm}
\newcolumntype{L}{>{$}l<{$}} 
\newcolumntype{C}{>{$}c<{$}} 

\newcommand{\Ri}[1]{#1}
\newcommand{\Rii}[1]{#1}
\newcommand{\sout}[1]{#1}

\newcommand{\tro}{\mathcal{L}}
\newcommand{\trr}{\mathcal{Q}}
\newcommand{\sch}{\mathcal{S}}
\newcommand{\prj}{\mathcal{P}}

\newcommand{\si}{\mathcal{K}}
\newcommand{\ac}{{\rho}}
\newcommand{\ti}{\mathscr{S}}
\newcommand{\spc}{E}
\newcommand{\osp}{{V_\perp}}

\newcommand{\zsp}{{V}}
\newcommand{\N}{\mathbb{N}}
\newcommand{\R}{\mathbb{R}}
\newcommand{\Z}{\mathbb{Z}}

\newcommand{\V}{\sigma^2}
\newcommand{\D}{\mathrm{D}}
\newcommand{\DD}{\mathrm{DD}}
\newcommand{\DA}{\mathrm{AD}}
\newcommand{\M}{\mathcal{M}}
\newcommand{\E}{\mathcal{E}}

\newcommand{\upd}{U_{P}}
\newcommand{\unp}{U_{NP}}
\newcommand{\uunp}{U^u_{NP}}
\newcommand{\bupd}{\bar U_{P}}
\newcommand{\bunp}{\bar U_{NP}}

\newcommand{\brf}{^{(f)}}
\newcommand{\brg}{^{(g)}}

\newcommand{\Pa}{p}
\newcommand{\Pb}{q}
\newcommand{\Pc}{r}

\newcommand{\as}{\leftarrow}

\newcommand{\Or}{O}
\renewcommand{\O}{\mathcal{O}}
\renewcommand{\upsilon}{\nu} 

\DeclareMathOperator{\id}{id}
\DeclareMathOperator{\midp}{midpoint}

\DeclareMathOperator*{\Lip}{Lip}
\DeclareMathOperator*{\spn}{span}
\DeclareMathOperator*{\sgn}{sgn}

\usepackage{amsfonts}
\allowdisplaybreaks

\usepackage{amsthm,amsaddr}
\newtheorem{lemma}{Lemma}
\newtheorem{theorem}{Theorem}
\newtheorem{corollary}{Corollary}
\newtheorem{remark}{Remark}

		\title[Spectral Galerkin methods for transfer operators]{Spectral Galerkin methods for transfer operators in uniformly expanding dynamics}

		\author{Caroline Wormell}
		\address{School of Mathematics and Statistics, University of Sydney, NSW 2006, Australia}
		\email[C. Wormell]{ca.wormell@gmail.com} 
\begin{document}
		\begin{abstract}
		Markov expanding maps, a class of simple chaotic systems, are commonly used as models for chaotic dynamics, but existing numerical methods to study long-time statistical properties such as invariant measures have a poor trade-off between computational effort and accuracy. We develop a spectral Galerkin method for these maps' transfer operators, estimating statistical quantities using finite submatrices of the transfer operators' infinite Fourier or Chebyshev basis coefficient matrices. Rates of convergence of these estimates are obtained via quantitative bounds on the full transfer operator matrix entries; we find the method furnishes up to exponentially accurate estimates of statistical properties in only a polynomially large computational time. 

To implement these results we suggest and demonstrate two algorithms: a rigorously-validated algorithm, and a fast, more convenient adaptive algorithm. Using the first algorithm we prove rigorous bounds on some exemplar quantities that are \Rii{substantially} more accurate than previous. We show that the adaptive algorithm can produce double floating-point accuracy estimates in a fraction of a second on a personal computer.  
		\end{abstract}

	\maketitle

\section{Introduction}\label{s:introduction}

One-dimensional full-branch Markov uniformly expanding maps are an important class of chaotic dynamical systems: as well as being common toy models, complex chaotic systems may be reduced to this class\Rii{, \sout{for example by inducing}}. Mathematically, these maps are endomorphisms $f$ on a compact and connected one-dimensional manifold $\Lambda$ for which there are a set of disjoint open intervals $(\O_\iota)_{\iota\in I}$ of full measure such that the $f|_{\O_\iota}$ are injective with $\overline{f(\O_\iota)} = \Lambda$, and on the $\O_\iota$, $f$ is differentiable with $ |f'| \geq \lambda > 1$. A standard example of such a map is the Lanford map, defined on $[0,1]$ with $f(x) = 2x + \frac{1}{2} x (1-x) \mod 1$.

Many significant properties of these maps can be determined from a \Rii{linear-algebraic} object: the so-called transfer operator $\tro: BV(\Lambda) \to BV(\Lambda)$ with action
\begin{equation} (\tro \phi)(x) = \sum_{f(y) = x} \frac{1}{|f'(y)|} \phi(y), \label{transfer_action}\end{equation}
where $BV(\Lambda)$ denotes the space of functions of bounded variation on $\Lambda$.
For example, the central limit theorem for the long-time average of an observable $\phi \in BV(\Lambda)$ for initial condition $x_0$ sampled from a $BV$ density can be written
\[ \frac{\sum_{i=0}^n \phi(f^i(x_0)) - n \langle \phi \rangle}{\sigma_f(\phi) \sqrt{n}} \xrightarrow{n\to\infty}_d N(0,1), \]
with formulae for the parameters
\begin{align} \langle \phi \rangle &:= \int_{\Lambda} \phi\, \rho\, dx\\
\sigma^2_f(\phi) &:= \int_{\Lambda} \phi \, \sum_{n=-\infty}^\infty \tro^{|n|} \big(\left(\phi - \langle \phi \rangle\right)\, \rho\big)\, dx, \label{birkhoff_variance_sum}
\end{align}
where the so-called invariant density $\rho$ is the unique 1-eigenfunction of $\tro$ with $\int_{\Lambda} \rho\, dx = 1$.

Transfer operator problems cannot in general be solved analytically, and numerical approaches are therefore of prime importance. One scheme that has been widely studied in the literature is Ulam's method, whereby one projects the transfer operator onto a subspace of characteristic functions (i.e. discretises the phase space) and computes statistical properties on this discretisation \cite{GAIO}. Ulam's method is effective for a broad array of families of chaotic systems \cite{Froyland11,Froyland07,Murray10}, and in particular Ulam estimates for a variety of statistical quantities have been proven to converge for uniformly expanding maps \cite{Froyland07, Bahsoun16}. Higher-order generalisations of Ulam's method have also been used, in particular to compute quantities such as linear response that require a higher degree of regularity \cite{Bahsoun18}; theory for a wavelet-based method has also been developed \cite{Holschneider96}. At the same time, Pollicott, Jenkinson and others have presented a completely different approach, wherein one computes statistical properties using the theory of dynamical zeta functions: this involves computing sums over periodic orbits of the system \cite{Jenkinson05,Jenkinson17}. Zeta function-based methods have furnished the most accurate estimates in the literature.

The approach we take in this paper is to construct a so-called spectral Galerkin approximation, whereby one considers the transfer operator \Rii{of the dynamical system of interest} and functions it acts on in a basis of orthogonal polynomials and restricts to finite-dimensional spaces $E_N \leq BV$ spanned by low-index elements in the orthogonal basis. 

We consider the Fourier exponential basis $e_k(x) = e^{ikx},\, k \in \Z$, which is orthogonal in $L^2([0,2\pi))$, and the Chebyshev polynomial basis $T_k(x) = \cos k \cos^{-1} x,\, k \in \N$, which is orthogonal on $[-1,1]$ with respect to the weight $(1-x^2)^{-1/2}$.
 
In our theoretical results we find that, providing the maps under consideration exhibit sufficient regularity, these kinds of spectral methods provide up to exponentially fast convergence with a small numerical outlay. 
Our main theoretical results are that spectral Galerkin estimates of acims and the $1$-resolvent converge exponentially fast in the approximation order $N$ for analytic maps, and as $\Or(N^{2.5-r})$ for $C^r$ maps (we will be more specific about the kinds of map we consider in Section \ref{ss:setup}). The algorithmic outlay of our method is $\Or(N^3)$. These results are summarised respectively in Corollaries \ref{c:acim} and \ref{c:sum} in Section \ref{ss:theorems}. 

We obtain these results by defining a so-called solution operator that allows one to access transfer operator resolvent data at eigenvalue $1$ (Theorem \ref{t:solution} in Section \ref{ss:theorems}), and then showing that spectral Galerkin approximations of this solution operator converge at the aforementioned rates (Theorem \ref{t:main}). These rates of convergence are determined via bounds on entries of the spectral basis matrix representations of transfer operators, proved in Theorems \ref{t:entry_p}-\ref{t:entry_np}.

This theoretical work carries out some of the directions for further research suggested in \cite{Baladi99}, which proved convergence of eigenvalue and eigenvector estimates of transfer operators of circle maps in a wavelet basis (a transformation of a Fourier basis). In particular, we extend from periodic intervals to non-periodic intervals, and establish quantitative convergence rates for the invariant density and resolvent data at the eigenvalue $1$.

To illustrate the power of spectral methods, we apply a rigorously justified spectral method to the Lanford map to get bounds for the Lyapunov exponent and diffusion coefficient to 123 decimal places (see Theorem \ref{t:rig} in Section \ref{ss:theorems}).

We also demonstrate that adaptive spectral methods allow for very fast and user-friendly computation of statistical properties, via an implementation of transfer operator spectral methods in the Julia package Poltergeist. Using Poltergeist, the quantities in Theorem \ref{t:rig} may be estimated in under $0.1$ seconds on a personal computer to 13 decimal places of accuracy; the package also allows for computation of a great many statistical quantities not included here.

The rates of convergence we obtain compare very favourably with other approaches. While set-based approaches (Ulam's method) cover a much larger class of maps than we consider, they have an optimal convergence rate of only $\Or(\log(N)/N)$ irrespective of regularity, where $N$ is the size of the Ulam matrix. Spectral methods are also significantly more efficient than algorithms that use periodic orbits to calculate statistical quantities: in the case of analytic maps, these algorithms converge superexponentially in the order (i.e. maximum length of periodic orbits), but the number of periodic orbits that must be computed, and hence the computational cost, grows exponentially with the order \cite{Jenkinson17,Jenkinson05}. In terms of computational power $P$, which is the relevant quantity in practice, periodic orbit algorithms have error $\Or(e^{-k(\log P)^2})$ as opposed to the spectral method's convergence rate of $\Or(e^{-kP^{1/3}})$.

Consequently, the illustrative numerical bounds we obtain in Theorem \ref{t:rig} are well beyond the practicable capabilities of other numerical results. The previous best rigorous bound on the Lyapunov exponent found in the literature was $L_{exp} = 0.6575 \pm 0.0015$ obtained by Ulam's method in \cite{Galatolo14}; the diffusion coefficient was calculated in \cite{Bahsoun16} to less than one significant figure of accuracy and \cite{Jenkinson17}, which had an estimate correct to 55 significant figures, of which 17 were rigorously validated. For comparison, we use a comparable amount of computational power to obtain 127 and 123 validated significant figures respectively (see Section \ref{ss:rigorous} for further details).

Similarly, our adaptive algorithms provide a much higher degree of accuracy for practical uses than previous algorithms. This is of great use in broader study of chaotic dynamics, as having rigorous (or at least very reliable) algorithms at one's disposal allows one to easily explore mathematical phenomena for which analytical results may not yet exist. An example of such an endeavour is in \cite{Gottwald16}, where the authors make use of a Fourier spectral method to explore a particular rate of convergence in linear response theory, directing a subsequent proof. We hope that our spectral methods become a useful tool in theoretical and numerical study of chaotic systems.

There are several further directions for research. Numerical results indicate that the actual rates of convergence are slightly better than what we prove in this paper (see Section \ref{ss:adaptive}), and a different theoretical approach may yield the optimal convergence rates. While our paper is set in function spaces with bounded variation norms, most of our results are largely agnostic to the function space used. It would in particular be useful for justifying linear response estimates to prove convergence of the spectral method on appropriate scales of spaces. While the convergence of eigenvalues and eigenvectors was proven (without quantitative bounds) in \cite{Baladi99}, we also have some numerical evidence for convergence of dynamical determinant estimates. Our results may be extended to higher dimensions, possibly including maps with contracting directions.
\Ri{A significant problem with extending to higher dimensions, however, is that the number of basis functions necessary to compute estimates of a given accuracy will increase exponentially with the dimension: this may be remediable to a limited extent by using bases of smooth, compactly-supported wavelets \cite{Holschneider96, Baladi99}, which could lower the complexity as a result of their sparse structure.} Finally, by constructing efficient numerical inducing schemes, it seems likely that our methods can provide fast and accurate estimates of statistical properties for almost all major classes of one-dimensional chaotic maps, such as non-Markov expanding maps, intermittent maps and quadratic maps.

Our paper is structured as follows. In Section \ref{s:setuptheorems} we define the classes of maps pertinent to our results, and introduce the main theorems. In Section \ref{s:algorithms} we describe the algorithms we use that demonstrate the possiblities of transfer operator spectral methods, and in Section \ref{s:numerics} we give illustrative results from these algorithms. Finally, in Section \ref{s:proofs} we prove the theoretical results of the paper.

\section{Set-up and main theorems}\label{s:setuptheorems}

In this section, we summarise the paper's chief mathematical developments. We will first set up the problem, introducing the maps under consideration.
We then present the important results of the paper: Theorem \ref{t:solution} characterising an operator that explicitly solves many typical transfer operator problems; the main theorem, Theorem \ref{t:main}, which gives convergence of spectral operator estimates and Corollaries \ref{c:acim} and \ref{c:sum}, which give convergence of acims and other statistical properties as a result; and finally Theorems \ref{t:entry_p} and \ref{t:entry_np} bounding the magnitude of transfer operator spectral coefficient matrix entries $L_{jk}$, which are central to the proof of Theorem \ref{t:main}. We finally present two rigorously validated bounds on the Lyapunov exponent and a diffusion coefficient of the Lanford map, obtained via a rigorous implementation of our spectral method.

\subsection{Systems under consideration}\label{ss:setup}

We first introduce the two generic classes of maps we will consider; we will then introduce a set of so-called distortion conditions that maps from these classes may optionally hold, and which determine the spectral method's rates of convergence.

\subsubsection{Classes of maps}

We define two main classes of maps: \Rii{circle maps $\upd$ and interval maps $\unp$. Maps in $\upd$ are defined on the one-dimensional torus and must be continuous and differentiable on the whole domain, whereas maps in $\unp$ are defined on a (non-periodic) interval, and there is no requirement for any continuity or differentiability between branches of the map. For example, a Markovian tent map may lie in $\unp$, whereas maps in $\upd$ must have a derivative defined everywhere.}

A map $f: \Lambda \to \Lambda$ is in $\upd$ if it satisfies the following axioms:
\begin{itemize}
	\item Its domain $\Lambda$ is a \Rii{circle}, which we suppose to be canonically $\R / 2\pi\Z$.
	\item It is \Rii{\sout{piecewise}} $C^2$ with Lipschitz-bounded distortion, that is,
	\[ \sup_{x\in\Lambda} \frac{|f''(x)|}{|f'(x)|^3} < \infty.\]
	\item It is uniformly expanding, that is,
	\begin{equation} \lambda := \inf_{x\in\Lambda} |f'(x)| > 1. \label{ue_definition}\tag{E}\end{equation}
\end{itemize}
Maps in $\upd$ are circle maps, and can be extended to bijective lifts $\hat f: [0,2\pi] \to [0,2\beta\pi]$ for some $\beta \in \{2,3,4,\ldots\}$. We denote the inverse of $\hat f$ by $v$, and for consistency with the notation for $\unp$ define $v_\iota(x) := v(x + 2\iota\pi)$ for $x \in [0,2\pi]$ and $\iota\in I:= \{0,1,\ldots,\beta-1\}$.

A map $f: \Lambda \to \Lambda$ is in $\unp$ if it satisfies the following axioms:
\begin{itemize}
	\item Its domain $\Lambda$ is \Rii{an interval}, which we suppose to be the canonical interval for Chebyshev expansions $[-1,1]$.
	\item It is full-branch Markov: recall that this means there are open disjoint intervals $\O_\iota, \iota \in I$ whose union is of full measure in $\Lambda$ such that $f|_{\O_\iota}$ extends continuously to a bijective function $\hat f_\iota: \overline{\O_\iota} \to \Lambda$. 
	\item \Rii{These functions $\hat f_\iota$ are all $C^2$ and furthermore, the map $f$ has Lipschitz-bounded distortion: this is a standard regularity condition necessary for, among other things, a spectral gap in $BV$. We will find it useful to formulate the Lipschitz distortion condition in terms of $v_\iota := \hat f_\iota^{-1}$, as }
	\begin{equation} \sup_{x\in\Lambda, \iota \in I} \left|\frac{v_\iota''(x)}{v_\iota'(x)}\right| < \infty.\tag{$\DD_1$}\label{distortion_basic}\end{equation}
	\item It satisfies a {\it uniform C-expansion condition}\footnote{This condition can be reformulated as requiring $|(\cos^{-1}\circ f \circ \cos)'| \geq \check\lambda > 1$. }
		\begin{equation} \check\lambda = \inf_{x\in\cup_{\iota\in I} \O_\iota} \frac{\sqrt{1-x^2}}{\sqrt{1-f(x)^2}} |f'(x)| > 1. \tag{CE} \label{aue_definition} \end{equation}
	\item It satisfies a {\it partition spacing condition}
\begin{equation} \sup\left\{ \frac{|\O_\iota|}{d(\O_\iota,\partial \Lambda)} : d(\O_\iota,\partial \Lambda) > 0 \right\} = \Xi < \infty. \label{ppc}\tag{P}\end{equation}
\end{itemize}
The latter two conditions we introduce to control the high oscillatory behaviour of the spectral basis functions' \Rii{images under the action of the transfer operator} near the endpoints of the interval. They are not especially onerous conditions: \Rii{uniformly expanding maps typically satisfy (\ref{aue_definition})}, and a uniformly expanding map satisfying all conditions of $\unp$ except (\ref{aue_definition}) will have an iterate in $\unp$ (see Appendix \ref{a:expanding-conjugacy} for a discussion of C-expansion); $(\ref{ppc})$ is always satisfied for maps with finitely many branches.

We also consider maps that satisfy the conditions of $\upd$ (resp. $\unp$) except that the associated expansion parameter in $(\ref{ue_definition})$ (resp. (\ref{aue_definition})) need only be positive, rather than strictly greater than $1$. We denote the class of such maps $\bupd$ (resp. $\bunp$).


\subsubsection{Distortion conditions}

To obtain good convergence results we will optionally impose the following generalised distortion conditions on our maps.

The first set of distortion conditions are equivalent to uniform bounds on derivatives of the distortion $\log |v_\iota'|$. A map satisfies distortion condition (\ref{distortion_diff_frac}) for some $r \in \N^+$ if
\begin{equation} \label{distortion_diff_frac}  \sup_{\iota\in I,x\in\Lambda}\left|\frac{v_\iota^{(n+1)}(x)}{v_\iota'(x)}\right|= C_n < \infty,\ n = 1, \ldots, r. \tag{$\DD_r$} \end{equation}

The second set of distortion conditions are equivalent to uniform bounds on the first derivative of the distortion on a complex neighbourhood of the map's domain $\Lambda$. For circle maps, the neighbourhood is the closed complex strip
$\Lambda^\beta_\delta = \{x+iy \mid  x \in \R/2\beta\pi\Z, |y| \leq \delta\},$
for a given $\delta > 0$. For intervals, the neighbourhood is $\check{\Lambda}_\delta$, defined to be a Bernstein ellipse\footnote{A Bernstein ellipse of parameter $\rho>1$ is an ellipse in the complex plane centred at $0$ with semi-major axis of length $\frac{1}{2}(\rho + \rho^{-1})$ along the real line and semi-minor axis $\frac{1}{2}(\rho - \rho^{-1})$.} of parameter $e^{\delta}$. We assume that $v'$ and $v'_\iota$ respectively extend holomorphically to these sets.

A map satisfies (\ref{distortion_ana}) for some $\delta > 0$ if
 \begin{equation} 
 \begin{cases}\sup_{z\in\Lambda^\beta_\delta}\left|\frac{v''(z)}{v'(z)}\right| = C_{1,\delta} < \infty,& \Lambda = \R/2\pi\Z,\\
 \sup_{\iota\in I,z\in \check{\Lambda}_\delta}\left|\frac{v_\iota''(z)}{v_\iota'(z)}\right| = C_{1,\delta} < \infty,& \Lambda = [-1,1].
 \end{cases}
 \label{distortion_ana}\tag{$\DA_\delta$}\end{equation}
 
 We associate with each distortion condition a {\it spectral rate of convergence}. We formulate these rates of convergence $\kappa(\cdot)$ as function classes:
 \begin{align}
 \kappa\left(\DD_{r}\right) &= \left\{ x \mapsto C (1+x)^{-r} : C > 0\right\}, \label{spectral_rate_diff}\\ 
 \kappa\left(\DA_{\delta}\right) &= \left\{x \mapsto C e^{-\zeta x} : C > 0,\ \zeta \in (0,\delta] \right\}. \label{spectral_rate_ana}\end{align}

\subsection{Main results}\label{ss:theorems}


We can now formulate the main theoretical results of this paper, beginning by introducing a novel operator derived from the transfer operator that explicitly generates acims and other statistical properties. 

We define the {\it solution operator inverse} 
\begin{equation}\si = \id - \tro + u\ti\label{solution_operator_inverse}\end{equation}
and the {\it solution operator}
\begin{equation} \sch = \si^{-1} = (\id - \tro + u \ti)^{-1},\label{solution_operator}\end{equation}
where the functional $\ti$ is the \Rii{\sout{total}} Lebesgue integral on $\Lambda$ and $u$ is a \Rii{function in the domain of $\tro$} such that $\ti u = 1$. 

Many statistical properties can be computed using resolvent data of $\tro$ at its eigenvalue $1$: the solution operator inverse is a \Rii{bounded}, invertible perturbation of $\id - \tro$ which allows the resolvent data to be recovered.
For any transfer operator $\tro$ with a spectral gap (i.e. with a simple eigenvalue at $1$ and the remaining spectrum bounded inside a disk of radius less than $1$), the solution operator therefore solves for two important quantities, according to the following theorem:
\begin{theorem}\label{t:solution}
	Let $\tro: \spc \to \spc$ be a transfer operator with a spectral gap. Choose $u \in \spc$ with $\ti u = 1$.
	
	
	Then $\sch = (\id - \tro + u \ti)^{-1}$ is well-defined and bounded as an operator on $\spc$, and
	\begin{enumerate}[(a)]
		\item If $\ac$ is the unique acim with $\ti\ac=1$,
		\begin{equation} \ac = \sch u. \label{acim_solution} \end{equation}
		\item  For any $\phi \in \ker \ti$,
	\begin{equation} \sum_{n=0}^\infty \tro^n \phi = \sch \phi. \label{csum_solution} \end{equation}
		\end{enumerate}
\end{theorem}

\begin{remark}\label{r:variance}
	As a  result of Theorem \ref{t:solution}, many important statistical quantities can be simply expressed using the solution operator and acim.
	For example, the Green-Kubo formula for diffusion coefficients given in (\ref{birkhoff_variance_sum}) can be rewritten using Theorem \ref{t:solution}(b) as
	\begin{equation} \V_f(A) =  \int_\Lambda A\ (2\sch-\id)(\id - \rho \ti)(\rho A)\,dx. \label{birkhoff_variance}\end{equation}
	This closed formula enables effective rigorous calculation of diffusion coefficients.
\end{remark}


We now provide some notation to enable us to state the main theorem, which proves the convergence of the spectral methods. We define the finite-dimensional subspaces $(E_N)_{N\in\N^+}$
\[ E_N = \begin{cases} \spn\{e_{-N},\ldots,e_N \}, \Lambda = [0,2\pi) \\ \spn\{T_0,\ldots,T_N\}, \Lambda = [-1,1]
\end{cases}\]
and the orthogonal projections $\prj_N$ onto the $E_N$ in the \Rii{$L^2$ space in which the bases are orthogonal ($L^2([0,2\pi])$ for the Fourier basis, and $L^2([-1,1],(1-x^2)^{-1/2})$ for the Chebyshev basis)}. We also define the spectral Galerkin operator discretisations
\begin{equation} \tro_N = \prj_N \tro |_{\spc_N}\label{transfer_operator_finite}\end{equation}
and
\begin{equation} \sch_N := \si_N^{-1} := (\id - \tro_N + u \ti|_{\spc_N})^{-1},\label{solution_operator_finite}\end{equation}
where the function $u$ is taken to be in $E_N$. (A typical choice of $u$ is $u = 1/|\Lambda|$.)


Our main theorem can then be formulated as follows:
\begin{theorem}\label{t:main}
	Suppose $f \in \upd$ or $\unp$, and satisfies a distortion bound 
	$($\ref{distortion_diff_frac}$)$ (resp. $($\ref{distortion_ana}$)$).
	Then there exist functions $K,\bar K \in 
	\kappa(\DD_r)$ (resp. $\kappa(\DA_\delta)$)
	such that for sufficiently large $N$ and all $\phi \in E_N$,
	\[ \| \tro_N \phi - \tro \phi \|_{BV} < N\sqrt{N} K(N) \| \phi\|_{BV}, \]
	and
			\[ \| \sch_N \phi - \sch \phi \|_{BV} < N\sqrt{N} \bar K(N) \| \phi\|_{BV}. \]
\end{theorem}

For ease of expression, in the rest of this section we use the notation ($\D$) to denote either of (\ref{distortion_diff_frac}) or (\ref{distortion_ana}). 

Theorem \ref{t:main} together with Theorem \ref{t:solution} directly implies the convergence of estimates of statistical quantities.
In particular, the following corollary gives spectral convergence of the acim.
\begin{corollary}\label{c:acim}
	Suppose $f\in\upd$ or $\unp$, and satisfies a distortion bound $(\D)$. 
	
	
	Let $\ac_N = \sch_N u$.
	Then there exists $K \in \kappa(\D)$ such that for all $N$ sufficiently large
	\[ \| \ac_N - \ac \|_{BV} < N\sqrt{N} K(N). \]
\end{corollary}

The next corollary gives strong convergence of $\sum_{n=1}^\infty \tro^n$, and consequently many important statistical estimates (see (\ref{birkhoff_variance}) for an example).

\begin{corollary}\label{c:sum}
	Suppose $f\in\upd$ or $\unp$, and satisfies a distortion bound $(\D)$. 
	 Then there exists $K\in \kappa(\D)$ such that for $N$ large enough and all $\phi \in E_N \cap \ker\ti$, 
	\[ \left\| \sch_N \phi - \sum_{n=0}^\infty \tro^n \phi \right\|_{BV} < N\sqrt{N} K(N) \| \phi\|_{BV}. \]
\end{corollary}

Since the operators $\tro_N$ and $\sch_N$ are endomorphisms on $\spc_N$, Theorem \ref{t:main} and Corollary \ref{c:sum} show that the spectral method converges in operator norm within $\spc_N$. When attempting to estimate, for example, $\sch \phi$ for some $\phi \notin \spc_N$, one can simply substitute $\phi$ for its spectral discretisation $\prj_N \phi$, and propagate  through the calculation the error arising from this substitution.

Critical to proving Theorem \ref{t:main} are the following bounds on the entries $L_{jk}$ of the transfer operator matrix. We state two analogous theorems for transfer operators on periodic and non-periodic domains: the situation is illustrated in Figure \ref{f:theorems}. Abstractly, these results reformulate the characterisation of the transfer operator of a uniformly-expanding map as the sum of a strictly upper-triangular operator and a compact operator developed by \cite{Holschneider96,Baladi99} in the context of $C^\infty$ circle maps in wavelet bases. The important development of our approach is the large amount of quantitative information generated, which allows us to prove convergence rates and provide rigorous concrete bounds for specific maps. 

\begin{theorem}\label{t:entry_p}
	Suppose $f$ is in the class $\bupd$  satisfying some distortion bound $(\D)$, with $\lambda_1\leq f'\leq \lambda_2$. Suppose $L$ is the matrix representing the transfer operator of $f$ in a Fourier exponential basis.

	Then for every $p_1 > \lambda_1^{-1}$ and $p_2 <\lambda_2^{-1}$ there exists $K \in \kappa(\D)$ such that for $j/k > p_1$ or $k = 0$,
	$$ |L_{jk}| \leq K(|j - p_1 k|), $$
	and for $j/k < p_2$ or $k=0$,
	$$ |L_{jk}| \leq K(|j - p_2 k|). $$
	
\end{theorem}

\begin{figure}[htb]
	\centering
	\includegraphics[width=122mm
	]{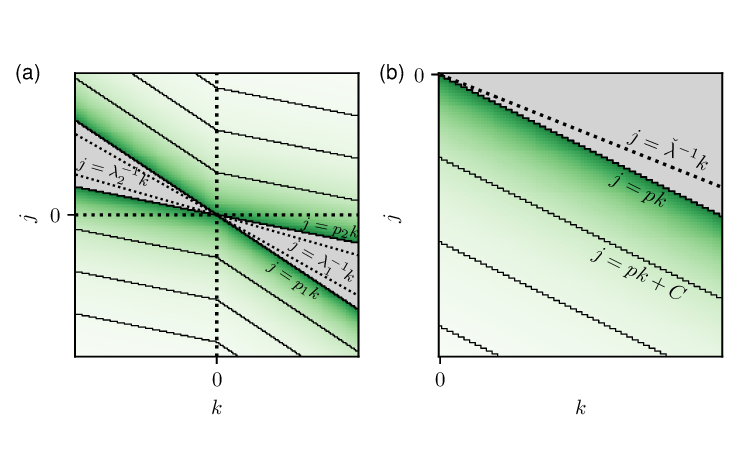}
	\caption{Heatmaps of maximum possible magnitudes of coefficients $L_{jk}$ of transfer operator matrix of a system described in: (a) Theorem \ref{t:entry_p}; (b) Theorem \ref{t:entry_np}. Shown are contours of constant magnitude of coefficients (thin black lines) and, in grey block colour, coefficients not characterised by the theorems.
		Note that because the full Fourier and Chebyshev bases are indexed by $\Z$ and $\N$ respectively, the matrix indices range over these values.} 
	\label{f:theorems}
\end{figure}


\begin{theorem}\label{t:entry_np}
	Suppose $f$ is in the class $\bunp$,  satisfying some distortion bound $(\D)$.  Suppose $L$ is the matrix representing the transfer operator of $f$ in a Fourier exponential basis. 
	
	Then for every $p>\check{\lambda}$ there exists $K \in \kappa(\D)$ such that for $j/k > p$ or $k = 0$,
	$$ |L_{jk}| \leq K( |j - p k|). $$
\end{theorem}

\begin{remark}
	One can prove similar results for transfer operators with general weights (c.f. (\ref{transfer_action})):
	\begin{equation} (\tro_g \phi)(x) = \sum_{f(y) = x} g(y) \phi(y). \label{baladi_transfer}\end{equation}
	This class of operator includes transfer operators and composition operators $C_v: \phi \mapsto \phi \circ v$.
\end{remark}

From the previous results, we are able to prove extremely accurate rigorous bounds on maps in $\upd$ and $\unp$ satisfying sufficiently strong distortion conditions. In particular, we prove the following bound on the Lanford map:
\begin{theorem}\label{t:rig}
	Consider the Lanford map $f: [0,1]\to[0,1]$, $f(x) = 2x + \frac{1}{2} x (1-x) \mod 1$.
	\begin{enumerate}[(a)]
		\item The Lanford map's Lyapunov exponent $L_{exp} := \int_{\Lambda} \log |f'|\, \rho\, dx$ lies in the range
		\begin{equation*}\begin{split} L_{exp} 
		&= 0.657\ 661\ 780\ 006\ 597\ 677\ 541\ 582\ 413\ 823\ 832\ 065\ 743\ 241\ 069
		\\&\qquad 580\ 012\ 201\ 953\ 952\ 802\ 691\ 632\ 666\ 111\ 554\ 023\ 759\ 556\ 459
		\\&\qquad 752\ 915\ 174\ 829\ 642\ 156\ 331\ 798\ 026\ 301\ 488\ 594\ 89 \pm 2 \times 10^{-128}.\end{split}\end{equation*}
		
		\item The diffusion coefficient for the Lanford map with observable $\phi(x) = x^2$ lies in the range
		\begin{equation*}\begin{split} \V_f(\phi) 
		&= 0.360\ 109\ 486\ 199\ 160\ 672\ 898\ 824\ 186\ 828\ 576\ 749\ 241\ 669\ 997
		\\&\qquad 797\ 228\ 864\ 358\ 977\ 865\ 838\ 174\ 403\ 103\ 617\ 477\ 981\ 402\ 783
		\\&\qquad 211\ 083\ 646\ 769\ 039\ 410\ 848\ 031\ 999\ 960\ 664\ 7 \pm 6 \times 10^{-124}.\end{split}\end{equation*}
	\end{enumerate}
\end{theorem}
These bounds are derived in Section \ref{ss:rigorous}.

\section{Algorithms}\label{s:algorithms}

Our results suggest a variety of possible algorithms to capture, given a map, statistical properties that can be expressed as $\sch \phi$ for some $\phi$, such as acims (\ref{acim_solution}) and diffusion coefficients (\ref{birkhoff_variance}).
We present two possible algorithms a practitioner might wish to use to calculate invariant measures: one that gives rigorous bounds on statistical properties but is somewhat cumbersome for exploratory use, and one that gives accurate but non-validated estimates that is much more convenient to use. In this section we describe the two algorithms, and then explain how in both algorithms we calculate elements of the transfer operator matrix. We will demonstrate the algorithms in Section \ref{s:numerics}.

Algorithm \ref{alg:rigorous} is a traditional fixed-order spectral method, implemented in interval arithmetic. It requires as input the map inverses $v_\iota$ and their derivatives, a spectral order $N$, various bounds associated with elements of the transfer operator, and a bound on the norm of the solution operator\footnote{Available theoretical bounds typically scale exponentially with the distortion bound $C_1$ (see \cite{Korepanov16} and Appendix \ref{a:bv-norm}). However, at least in the analytic case, the spectrally fast convergence dominates the large theoretical bounds. It is only necessary to control floating-point error using an appropriately high numerical precision.	
	Alternatively and more generally, one may apply the approach of \cite{Galatolo14}.} $\sch$; it then outputs an estimate for the acim $\rho$ with a rigorously validated $BV$ error.

\begin{algorithm}
	\SetKwComment{Com}{\# }{}
	\KwIn{Map inverses and derivatives $v_\iota, v_\iota',\ \iota \in I$; spectral order $N$; aliasing bounds $A_{jk}^{(N)}$ for $j,k = 1,\ldots,N$; bound $b^\sch \geq \|\sch\|_{BV}$; bound $b^{\E_N} \geq \|\E_N\|_{BV}$ (see Lemma \ref{l:strong-convergence}); bounds $b^{L}_{jk} \leq |L_{jk}|$.}
	\KwOut{High-precision floating-point vector $\tilde{\ac}$ containing spectral coefficients of acim estimate; rigorous $BV$ error bound $\bar{\epsilon}_{\mathrm{obs}}$}
	Check that {$b^{\E_N} b^{\sch}<1$}: if this is not the case increase $N$\;
	Set the number of floating-point bits to be greater than $-\log_2(N^4 \ast b^{\E_N})$\;
	Initialise $N\times N$ matrix of intervals $L^{(N)}$\;
	\For{$k \as 1$ \KwTo $N$}{
		Calculate interpolant values $q^{(k,N)} = \{ \tro(b_k)(x_{l,N}) \}_{l=1}^N$ in interval arithmetic, using  (\ref{transfer_action})\;
		Calculate spectral coefficients of the interpolant $p^{(k,N)} = FFT(q^{(k,N)})$ ($DCT(q^{(k,N)})$ in the Chebyshev case)\;
		\For{{\sf j} $\as 1$ \KwTo $N$}{
			Calculate spectral coefficient matrix entry $L^{(N)}_{jk}$ as $q^{(k,N)}_j$ plus aliasing error $[-A_{jk}^{(N)},A_{jk}^{(N)}]$\;
			Refine interval estimate $L^{(N)}_{jk}$ by intersecting it with $[-b^L_{jk},b^L_{jk}]$\;
		}
	}
	Calculate $u^{(N)} = \{[\delta_{j0}/|\Lambda|,\delta_{j0}/|\Lambda|]\}_{j=1}^N$\;
	Calculate row vector of intervals $\ti^{(N)} = (\ti b_j)_{j=1}^N$ using standard formulae \cite{Trefethen13}\;
	Calculate the spectral coefficient matrix of $\sch_N^{-1}$, $K^{(N)} = I - L^{(N)} + \ti^{(N)} u^{(N)}$, where $I$ is an $N\times N$ identity matrix\;
	Calculate $\ac^{(N)} = K^{(N)} \backslash u^{(N)}$\label{rig_solve}\;
	Calculate $\tilde{\ac} = \{ \midp(\ac^{(N)}_j) \}_{j=1}^N$\;
	Calculate a bound $\bar\epsilon_{\mathrm{interval}} > \|\ac^{(N)}-\tilde{\ac}\|_{BV}$\;
	Calculate $\bar\epsilon_{\mathrm{finite}} = 1/(1/(b^{\E_N} b^{\sch})-1)$\;
	Calculate $\bar{\epsilon}_{\mathrm{obs}} = \bar\epsilon_{\mathrm{interval}} + \bar\epsilon_{\mathrm{finite}} $\;
	\caption{Rigorous algorithm to capture invariant measures.}\label{alg:rigorous}
\end{algorithm}

\begin{algorithm}
	\SetKwComment{Com}{\# }{}
	\KwIn{Map $f$; map derivative $f'$ (optional; may be calculated automatically using dual number routines \cite{Revels16}); tolerance $\epsilon$}
	\KwOut{Adaptive order $k_{\mathrm{opt}}$; floating-point vector $\tilde{\boldsymbol{\ac}}$ containing spectral coefficients of acim estimate $\tilde{\ac}_{N_{\mathrm{opt}}}$}
	\Com{Extendable vectors encode an infinite vector with finitely many non-zero entries, ragged matrices' columns are extendable vectors. These will encode infinite-dimensional objects approximating $u$, $\si$.}
	
	Initialise empty ragged matrix $H$, which will hold Householder vectors for row-reduction\;
	Initialise empty ragged matrix $\hat{K}$, which will hold row-reduced coefficients of solution operator inverse $\si$\;
	Calculate extendable vector $\hat{u} = (\delta_{j1}/|\Lambda|)_{j\geq 1}$ containing coefficients of $u$ that will be progressively row-reduced\;
	Set $k$, the number of columns of matrix $\hat{K}$, to be 0\;
	\Repeat( \Com*[h]{Loop between calculating columns of $\hat{K}$ and row-reducing}){$\max \{u_j\}_{j\geq k+1} \leq\epsilon/|\Lambda|^{-1}$ \Com*[h]{i.e.\ negligible benefit from larger k}}{ 
		Increment $k$ by $1$\;

		Set the interpolation order $M = 4$\;
		\Repeat( \Com*[h]{Calculating optimum order interpolant of $\tro b_k$}){the interpolant has converged according to the reasoning in \cite{Aurentz17}}{
			Set $M\leftarrow 2M$\;
			Calculate values of the interpolant $q^{(k)} = \{\tro(b_k)(x_{l,M})\}_{l=1}^M$ using (\ref{transfer_action}) with Newton iteration for the transfer operator\;
			Calculate spectral coefficients of the interpolant $p^{(k)} = FFT(q^{(k)})$ ($DCT(q^{(k)})$ in the Chebyshev case)\;
		}
		Calculate $\kappa^{(k)}$, which will become the $k$th column of $\hat{K}$, as an extendable vector $\{\delta_{jk} + \ti_k \delta_{j1}\}_{j\geq1} - p^{(k)}$, where $\ti_k=(\ti b_k)$ is calculated from Chebyshev and Fourier integral formulae \cite{Trefethen13}\;
		Apply previous Householder transformations encoded as column vectors of $H$ to $\kappa^{(k)}$\;
		Calculate Householder vector $h$ that will row-reduce $\kappa^{(k)}$ considered as the $k$th column of $\hat{K}$\;
		Apply  $h$ to $\kappa^{(k)}$\;
				Right-concatenate $\kappa^{(k)}$ onto $\hat{K}$\;
					\Com{Note $\hat{K}$ is row-reduced and so upper-triangular}
		Apply $h$ to $\hat{u}$\;
		Right-concatenate $h$ onto $H$\;
	}
		Set $N_{\mathrm{opt}} = k$\;
	Calculate $\tilde{\boldsymbol{\rho}} = \{\hat{K}_{jk}\}_{j,k=1}^{N_{\mathrm{opt}}} \backslash \{\hat{u}_j\}_{j=1}^{N_{\mathrm{opt}}}$ via backsolving\;

	\caption{Algorithm to capture invariant measures using adaptive interpolation and infinite-dimensional adaptive QR solver.} 
	\label{alg:adaptive}
\end{algorithm}

By contrast, Algorithm \ref{alg:adaptive} is an adaptive-order spectral method that is not rigorously validated: it uses an adaptive QR factorisation of the solution operator inverse $\si$ to solve the linear problem and test for convergence \cite{Olver13Fast,Hansen10}. It requires as input only an algorithm to calculate the map $f$ and outputs an estimate for $\ac$ whose error is not rigorously bounded but is of the order of $\|\sch\|_{BV} \epsilon^{1-\theta}$, where $\epsilon$ is the floating-point precision and $\theta$ is a small number depending on the order of differentiability of $f$.

Algorithm \ref{alg:adaptive} is extremely well-suited for numerical exploration. Because the only required input is the map itself, Algorithm \ref{alg:adaptive} requires a minimum of drudge work on the part of the user. It is typically also extremely fast:  just with a personal computer, Algorithm \ref{alg:adaptive} gives estimates of statistical quantities of a simple analytic map accurate to $14$ decimal places in less than one-tenth of a second  (see Section \ref{s:numerics}). Because our spectral methods are very accurate in an easily verifiable way, an adaptive, non-validated method is also highly reliable. We have consequently made an implementation of Algorithm \ref{alg:adaptive} available in the open-source Julia package Poltergeist \cite{Poltergeist}.

In the presentation of the algorithms and the following discussion we assume that the Fourier and Chebyshev spectral bases have been relabeled as $(b_k)_{k\in\N^+}$. We also implicitly assume that the Fourier exponential basis has been transformed to sines and cosines so that real functions have real spectral coefficients.

In both algorithms, one calculates $L_N$ by columns, using that the $k$th column of $L_N$ consists of the first $N$ spectral coefficients of $\tro b_k$. The most effective way to estimate these coefficients is by calculating an interpolant. 
The idea of this is as follows. Using (\ref{transfer_action}) one evaluates the function $\tro b_k$ at $N$ special interpolation nodes $x_{l,N}$: in the Fourier case these interpolation nodes are evenly-spaced on the periodic invterval (in the Chebyshev case respectively, Chebyshev nodes of the first kind) \cite{Trefethen13,Boyd01}. One then applies the Fast Fourier Transform (resp. Discrete Cosine Transform) to the vector $((\tro b_k)(x_{l,N}))_{l=1,\ldots, N}$. The resulting length-$N$ vector contains the spectral coefficients of the unique function $p^{(k,N)} \in E_N$ which matches $\tro b_k$ at the interpolation nodes. The so-called interpolant $p^{(k,N)}$ is a close approximation of $\tro b_k$: the difference between the $j$th spectral coefficient of $p^{(k,N)}$ and that of $\tro b_k$ (the so-called aliasing error) is guaranteed to be smaller than some bound $A_{jk;N}$. This bound can be determined from aliasing formulae standard in approximation theory \cite{Trefethen13} combined with bounds on higher-order spectral elements of $\tro b_k$ (e.g. from Theorems \ref{t:entry_p}-\ref{t:entry_np}).

These algorithms generalise very easily to other transfer operator problems of the form $\psi = \sch \phi$: see for example the formula for diffusion coefficients (\ref{birkhoff_variance}). This can be done by formulating the problem as $\si \psi = \phi$ and thus substituting $\rho$ and $u$ (when it is not constituting the solution operator) for $\psi$ and $\phi$ respectively in the algorithms.

\section{Numerical results}\label{s:numerics}

In Section \ref{ss:rigorous} we will prove some rigorous bounds on basic statistical properties of the Lanford map using the rigorous Algorithm \ref{alg:rigorous}; we will then demonstrate the adaptive Algorithm \ref{alg:adaptive} using the Lanford map and a non-smooth circle map, assessing the adaptive algorithm's accuracy and the spectral method's rate of convergence.

\subsection{Rigorous bounds on statistical quantities: the Lanford map}\label{ss:rigorous}
	
	The Lanford map, $f: [0,1]\to[0,1]$
	\[ f(x) = 2x + \frac{1}{2}x(1-x) \mod 1 \]
	is a common test case for rigorous estimation of statistical quantities of maps \cite{Galatolo14,Jenkinson17,Bahsoun16}. By linearly rescaling of $[0,1]$ onto $[-1,1]$ we can apply our spectral method to it.
	
		The Lanford map's uniform expansion parameter is $\lambda = \frac{3}{2}$ and its distortion bound (on $[0,1]$) is $C_1 = \frac{4}{9}$.  
		Applying (\ref{solution_bound}), we find that
		$ \| \sch\|_{BV} \leq 9235$.
		
		By considering explicit bounds \Rii{that will be given in Lemma \ref{l:abstract_entry}}, we chose $\zeta = \cosh^{-1} \frac{7}{4}$, as it is close to the optimal value for $\zeta$ given in Remark \ref{r:optimal-zeta}. We then used a symbolic mathematics package to show that as a result of Remark \ref{r:optimal-zeta}, 
		\begin{equation} |L_{jk}| \leq t_j \sqrt{7+\frac{\sqrt{33}}{2}} e^{\cosh^{-1}(4-\sqrt{6}) k - \cosh^{-1} \frac{7}{4} j}. \label{lanford-coef-bound}\end{equation}
	
	To calculate an estimate of the acim of this map, we implemented Algorithm \ref{alg:rigorous} with $N=2048$. We found the truncation error was $\|\E_{N}\|_{BV} \leq  6.75 \times 10^{-133}$, and chose the floating-point precision to be 512 significand bits. 
	
	Consequently, we obtained an acim estimate $\tilde \ac$ with the rigorously validated error bound
	\[ \|\tilde \ac - \ac\|_{BV} \leq 6.3 \times 10^{-129}. \]
	This estimate is plotted in Figure \ref{f:rigorous-1}. The Chebyshev coefficients of $\tilde \ac$ are available in \url{Lanford-acim.zip}.
	
	\begin{figure}[htb]
		\centering
		\includegraphics[width=80mm
		]{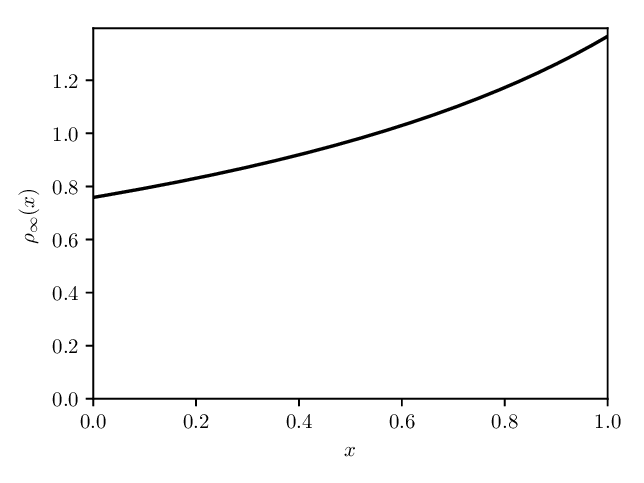}
		\caption{The density of the absolutely continuous invariant measure for the Lanford map, obtained by Algorithm \ref{alg:rigorous}.}
		\label{f:rigorous-1}
	\end{figure}
	
	We then used this estimate to calculate the Lyapunov exponent of the Lanford map
	\[ L_{exp} = \int_{\Lambda} \log |f'(x)| \ac(x) dx. \]
	using Clenshaw-Curtis quadrature on $\tilde \ac \log |f'| = \tilde \ac \log(2-3x)$ \cite{Trefethen13}. This provided the rigorous estimate given in Theorem \ref{t:rig}(a).

	We then calculated the diffusion coefficient of the observable $\phi(x) = x^2$ by evaluating the natural finite-order approximation of formula (\ref{birkhoff_variance}), using Clenshaw-Curtis quadrature.
	obtaining the rigorous bound given in Theorem \ref{t:rig}(b).

	The results together were obtained in 9 hours over 15 hyper-threaded cores of a research server running 2 E5-2667v3 CPUs with 128GB of memory. The most time-consuming operation was inverting the solution operator inverse matrix $\mathfrak{K}_{2048}$: this process took up $94\%$ of the runtime, which may stem partly from using an unoptimised routine. Once $\mathfrak{K}_{2048}^{-1}$, i.e. the solution operator matrix, was supplied, all the statistical quantities were calculated on a personal computer in seconds. 
	
\subsection{Adaptive algorithms}\label{ss:adaptive}

We now present results from the adaptive Algorithm \ref{alg:adaptive}, and illustrate the algorithm's convergence by comparison with a fixed-order version of Algorithm 2. 

We have implemented Algorithm \ref{alg:adaptive} in Julia, an open-source dynamic scientific computing language. This implementation is publically available in the package Poltergeist \cite{Poltergeist}. Poltergeist is integrated with ApproxFun, a comprehensive function approximation package written in Julia \cite{ApproxFun}; thus, standard manipulations of functions and operators may readily be applied to invariant measures, transfer operators and so on.

Using Poltergeist, we present empirical convergence results for the Lanford map (for comparison with rigorous methods), and a circle map which is $C^4$ but not analytic.

\subsubsection{The Lanford map}

The Lanford map experiment in Section \ref{ss:rigorous} can be repeated in Poltergeist in a few lines of Julia code:
\begin{verbatim}
using Poltergeist, ApproxFun
f_lift(x) = 5x/2 - x^2/2; d = 0..1
f = modulomap(f_lift,d);
K = SolutionInv(f); 
rho = acim(K);
L_exp = lyapunov(f,rho)
sigmasq_A = birkhoffvar(K,Fun(x->x^2,d))
\end{verbatim}

This code instantiates a {\tt MarkovMap} object {\tt f} 
and creates a {\tt QROperator} object \verb|K|, which stands in for the corresponding solution operator inverse $\si$ (recalling the definition of the solution operator inverse (\ref{solution_operator_inverse})). The \verb|acim| function carries out Algorithm \ref{alg:adaptive} by calling ApproxFun's adaptive QR solver \cite{Hansen10} on the equation \mbox{$\si\rho = u$}. The output is an ApproxFun {\tt Fun} object containing $\tilde{\boldsymbol{\ac}}_N$, the Chebyshev coefficients of the adaptive acim estimate. The Lyapunov exponent and diffusion coefficient are calculated using special commands defined in the package that call appropriate ApproxFun integration and QR solving routines, in the latter case via (\ref{birkhoff_variance}).
Once this the relevant functions have compiled using Julia's just-in-time compiler, the last five lines of the code will run in less than $0.12$ seconds on a personal computer. 

By applying Algorithm 2 with fixed orders $N$, the exponential convergence of $\ac_N$ with $N$ predicted in Theorem \ref{t:main} was seen to hold in practice. Indeed, only $N_{\mathrm{opt}} = 24$ columns of the transfer operator were required for convergence using Algorithm \ref{alg:adaptive} (see Figure \ref{f:adaptive-1}).

The Algorithm \ref{alg:adaptive} estimate for $\ac$ is in fact remarkably accurate: the $\ell^\infty$ error on Chebyshev coefficients is less than $8 \times 10^{-15}$ (40 times the floating point precision) and the BV error on the acim estimate is $3\times 10^{-13}$ (around 1300 times the floating point precision). The Lyapunov exponent estimate was correct almost to within the floating point precision, with the error compared to the rigorous estimate being $2.2 \times 10^{-16}$: this level of accuracy appears fortuitous rather than representative. More realistically, the estimate for $\V_f(A)$ was accurate to about $25$ times floating point precision ($1.4 \times 10^{-15}$).

	\begin{figure}[htb]
		\centering
		\includegraphics[width=80mm
		]{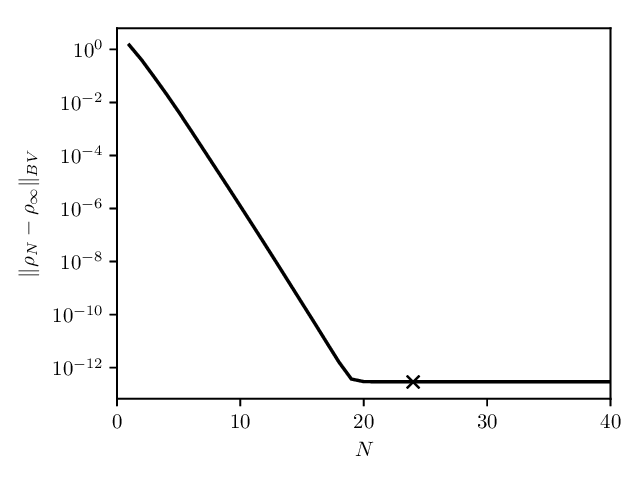}
		\caption{Exponential convergence with $N$ of floating-point estimates of $\ac_N$ for the Lanford circle map. The error of the adaptive estimate for $\ac_{N_{\mathrm{opt}}}$ from Algorithm \ref{alg:adaptive} is shown as a cross for comparison.
		}
		\label{f:adaptive-1}
	\end{figure}

\subsubsection{A non-analytic circle map}

We now consider a circle map which does not satisfy an analytic distortion condition (\ref{distortion_ana}) but rather a differentiable distortion condition (\ref{distortion_diff_frac}).

Define the uniformly expanding, triple-covering circle map $g: [0,2\pi) \to [0,2\pi)$ via the inverse of its lift:
\[ v_{g}(x) = \frac{x}{3} + \sum_{m=0}^{\infty} 2^{-\frac{33}{8} m} \cos{\left(2^m \left(1-\cos \frac{x}{3}\right)\right)}. \]
The map $g$ is $C^{4.125-\epsilon}$ and thus satisfies distortion condition $(\DD_3)$ but not $(\DD_4)$.

We implement the acim-finding process in a similar fashion to the Lanford map, although to optimise for speed we also supply \verb|CircleMap| with the derivative for the lift:
\begin{verbatim}
g = CircleMap(v_g,0..2pi,diff=v_g_dash,dir=Reverse)
Lg = Transfer(g);
rho_g = acim(Lg);
\end{verbatim}
This routine took approximately $9$ minutes to run on a personal computer and required the evaluation of $N_{\mathrm{opt}}=2747$ columns of the transfer operator. It produced an acim estimate (plotted in Figure \ref{f:adaptive-2}) whose $BV$ error we estimate to be approximately $4.8 \times 10^{-10}$, by comparison with an estimate obtained using high-precision floating-point arithmetic and $N=6144$ columns. 

	\begin{figure}[htb]
		\centering
		\includegraphics[width=80mm
		]{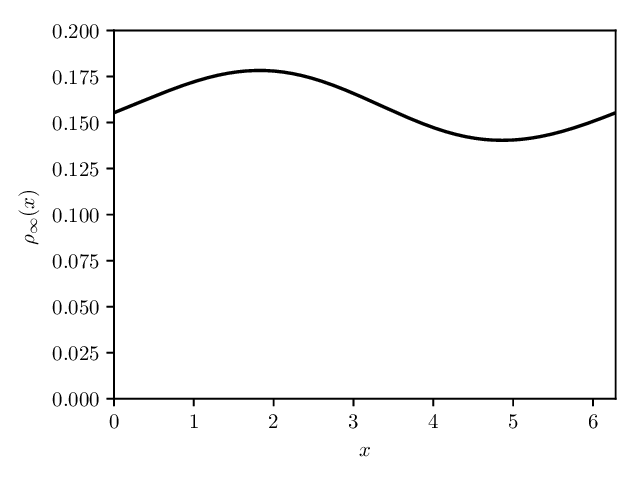}
		\caption{Invariant measure estimate for $g$ using Algorithm \ref{alg:adaptive}.}
		\label{f:adaptive-2}
	\end{figure}

	\begin{figure}[htb]
		\centering
		\includegraphics[width=80mm
		]{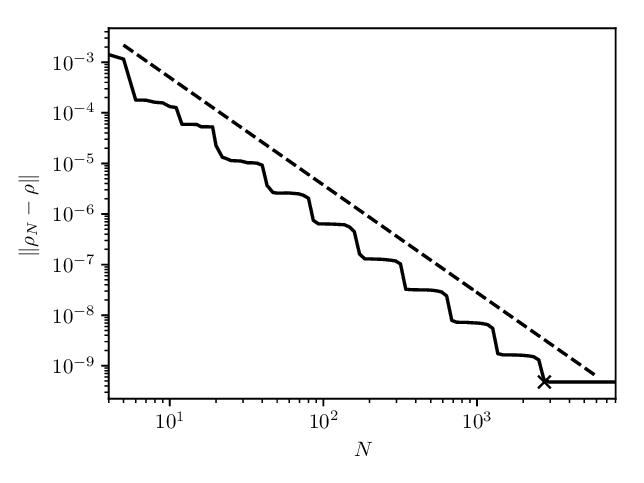}
		\caption{The convergence with $N$ of floating-point estimates of $\ac_N$ for $g$. 
			The error of the adaptive estimate for $\ac_{N_{\mathrm{opt}}}$ using Algorithm \ref{alg:adaptive} is plotted with a cross.
			The slope of a function $K(N) = C N^{-2.125}$ is plotted with a dashed line. Error estimates are by comparison with an $N=6144$ high-precision floating point acim estimate.}
		\label{f:adaptive-3}
	\end{figure}

The convergence of $\ac_N$ is illustrated in Figure \ref{f:adaptive-3}. The $BV$ error on $\ac_N$ is estimated to be $\Or(N^{\epsilon-2.125})$, which is better than the Theorem \ref{t:main} estimate of $\Or(N^{-1.5})$. 
We conjecture that acim estimates of $C^{r+\alpha}$ circle maps (i.e. those satisfying ``($\DD_{r-1+\alpha}$)'') converge in $BV$ as $\Or(N^{2-r-\alpha} \log N)$.

\begin{remark}
	Numerical experiments demonstrate that eigenvalues and eigenfunctions of $L_N$ converge in norm to those of $L$, as proved in the periodic case by \cite{Baladi99}. The observed rates of convergence are the standard spectral rates.
%
\end{remark}

\section{Proofs of results}\label{s:proofs}

Our attack on the theorems in Section \ref{ss:theorems} is structured as follows.

We begin by proving Theorem \ref{t:solution} characterising the solution operator. This proof uses standard linear-algebraic properties of transfer operators. 

We then turn to proving the entry bound results (Theorems \ref{t:entry_p} and \ref{t:entry_np}). These results stem from more general properties of Fourier series representations of composition operators (Lemma \ref{l:abstract_entry}), which we prove using oscillatory integral techniques. Because it is necessary to make a non-diffeomorphic cosine transformation to obtain Fourier basis functions from Chebyshev polynomials, some work is required to prove appropriate bounds on derivatives after the transformation.

We then go on to prove Theorem \ref{t:main}. We consider a perturbation of the transfer operator $\tro$ that is block-upper-triangular in the relevant spectral basis (in the Fourier case, under the basis order $e_0, e_1, e_{-1},e_2,e_{-2},\ldots$), with the finite matrix $L_N$ forming the first block on the diagonal. Since the solution operator of such a perturbation is a composition of upper block-diagonal operators, the first diagonal block can thus be approximated only from knowledge of $L_N$ (Lemma \ref{l:strong-convergence}). Using that the $BV$-norm of our perturbation can be bounded using spectral matrix coefficients (Lemma \ref{l:bvnorm}), we obtain the main result.
\\

We begin with the proof of Theorem \ref{t:solution}, which gives the properties of the solution operator $\sch = (I-\tro + u \ti)^{-1}$ (see (\ref{solution_operator})).

\begin{proof}[
	Theorem \ref{t:solution}]
	Split $\spc$ as $\osp \oplus \zsp$ where $\osp = \spn\{u\}$ and $\zsp = \ker\ti$.  \Rii{Since $\osp$ and $\zsp$ are closed subspaces of $\spc$ there exists a bounded operator $\mathcal{N}: \spc \to \osp$ such that one may also define $\id - \mathcal{N}: \spc \to \zsp$.} We now consider the action with respect to this splitting of the putative solution operator inverse, $\si = \id - \tro + u \ti$.
	
	Since $\ti(\tro - \id) = 0$, we have for any element $\phi \in \zsp$ that
	\begin{equation} \si \phi = (\id - \tro) \phi + u \ti \phi = (\id - \tro) |_\zsp \phi \in \zsp. \label{t-solution_u_action1} \end{equation}
	Similarly, for any scalar $\alpha$ we have
	\begin{equation} \si \alpha u = (\id-\tro) (\alpha  u) + u  \ti\alpha u =  \alpha u +  (\id - \tro)(\alpha u), \label{t-solution_u_action} \end{equation}
	where the sum follows the splitting of $\spc = \osp \oplus \zsp$.
	
	\Rii{
	Since the transfer operator $\tro$ has a spectral gap, the spectral radius of $\tro |_\zsp$ is strictly less than $1$, and the operator
	\[\trr := (\id-\tro|_{\zsp})^{-1} = \sum_{n=0}^\infty \left. \tro^n \right|_{\zsp}\] 
	is bounded as an endomorphism on $\zsp$.
	
	Back-solving (\ref{t-solution_u_action1}-\ref{t-solution_u_action}) thus gives that for any $\psi \in \zsp$,
		\[\sch \phi = \si^{-1} \psi = \trr \psi, \]
		and for any scalar $\alpha$ that
		\[ \sch \alpha u = \alpha u - \trr (\id - \tro) \alpha u = \lim_{n\to\infty} \tro^n \alpha u = \alpha \ac.\]
	
	Since $\alpha \rho = \rho \ti(\alpha u)$, we can use these results to write the solution operator
	\begin{equation} \sch = \trr (\id - \mathcal{N}) + \ac \ti \mathcal{N}, \label{t-solution-sum} \end{equation}
	which is clearly bounded.
	
	It clearly follows from (\ref{t-solution-sum}) that $\sch u = \ac$ and $\sch \phi = \trr \phi$ for $\phi \in \zsp$. }
\qed \end{proof}

\begin{remark}\label{r:solution_sum}
%
The solution operator can be written as the following expression
	\begin{equation} \sch = u \ti + \sum_{n=0}^{\infty} \tro^n (\id + (\tro u - 2u) \ti). \label{solution-sum}\end{equation}
\end{remark}

We now set about proving Theorems \ref{t:entry_p} and \ref{t:entry_np}, which place bounds on the magnitudes of the entries of transfer operator matrices in Fourier and Chebyshev bases. 

We begin by proving similar kinds of bounds on the coefficients of a matrix associated with a more general operator $\M$ on the  circle $\R/2\pi\Z$. $\M$ can be viewed as a generalised transfer operator (\ref{baladi_transfer}) where instead of using the inverse of the map, one uses a general function $v$ which may be non-injective. Bounds on elements of the Fourier basis transfer operator matrix for $\M$ imply similar bounds on transfer operators in Fourier and Chebyshev bases.

\begin{lemma}\label{l:abstract_entry}
	Let $v$ be a differentiable function from $\R/2\beta\pi\Z,\ \beta\in\Z^+$ to $\R/2\pi\Z$ such that $v'(\R/2\beta\pi\Z) = \tilde{\mu} = [\mu_2,\mu_1]$, and let $h$ be a continuous function on the circle $\R/2\beta\pi\Z$.
	
	Let $\M$ be \Rii{the} endomorphism on $L_2([0,2\pi])$ defined by
	\begin{equation} \mathcal{M}:\ \phi \mapsto \sum_{b=1}^{\beta} h\left(x+2\pi b\right) \phi\left(v\left(x+2\pi b\right)\right).\label{calm_defn} \end{equation}
	
	Let $M$ be the corresponding bi-infinite matrix in the Fourier complex exponential basis. 

	Then:
	\begin{enumerate}[(a)]
		\item The \Rii{entries} of $M$ are bounded uniformly by $\|h\|_1/2\pi$.
		
		\item Suppose that 
		for $n=1,\ldots,r$, $\sup|v^{(n+1)}| \leq \Upsilon_{n} < \infty$ and $\sup |h^{(n)}/h| \leq H_n < \infty$. Then there exist constants $W_{r,n}$ such that for $j\notin k \tilde{\mu}$,
		\begin{equation} |M_{jk}| \leq \frac{\|h\|_1}{2\pi} \sum_{n=0}^{r} \frac{W_{r,n} |k|^n}{d(j,k \tilde{\mu})^{n+r}}.\label{entry_differentiable} \end{equation}
		
		Each $W_{r,n}$ is bounded by a linear combination of $H_{l},\ l\leq r-n$, whose coefficients are polynomials in $\Upsilon_{l},\ l\leq r-n$.
		
		
		\item Suppose $v$ and $h$ extend analytically to the complex strip $\Lambda^\beta_\delta = [0,2\beta\pi) + i[-\delta,\delta]$, and on this strip $\sup |h'/h| \leq H_{1,\delta} <\infty$ and $\sup |v''| \leq \Upsilon_{1,\delta} < \infty$. 
		
		Choose any $\tilde p = [p_1,p_2]$ such that $\tilde{\mu} \subset \int\tilde p$.
		
		Define
		$\zeta = \min\left\{2\Upsilon_{1,\delta}^{-1} d(\tilde{\mu},\R\backslash \tilde{p}),\delta\right\}.$ 
	
		Then $\zeta > 0$ and
		\begin{equation} |M_{jk}| \leq \frac{\| h\|_1}{2\pi} e^{\zeta \left(H_{1,\delta} - d(j,\tilde p)\right)}.\label{entry_analytic} \end{equation}
	\end{enumerate}
\end{lemma}

%

\begin{proof}[
	Lemma \ref{l:abstract_entry}]
	The matrix element $M_{jk}$ is the $j$th Fourier coefficient of the function $\mathcal{M} e_k$, so using the orthogonality of Fourier bases in $L^2$ and (\ref{calm_defn}), we have that
	\[ M_{jk} = \frac{1}{2\pi} \int_{0}^{2\pi} \sum_{b=1}^{\beta} h\left(x+2\pi b\right) e^{ik v\left(x+2\pi b\right)} e^{-ijx}dx, \]
	which using the $2\pi$-periodicity of $e^{ijx}$ we can rewrite as a single integral
	\begin{equation} M_{jk} = \frac{1}{2\pi} \int_0^{2\beta\pi} h(x) e^{i(kv(x)- jx)} dx. \label{entry_integral}\end{equation}
	
	We obtain (a) from this equation simply by taking absolute values.
	
	For (b), we use that the integrand in (\ref{entry_integral}) is oscillatory when the derivative of $k v(x) - j x$ is bounded away from zero, that is, when $j/k \notin [\mu_2,\mu_1]$. As a result, we can improve the bound we got in the first part by repeatedly integrating by parts.
	
	Starting from (\ref{entry_integral}), we separate the integrand into two terms
	\[ \left(\frac{h(x)}{i(kv'(x)-j)}\right) \left(i(kv'(x)-j) e^{i(kv(x)-jx)}\right), \]
	so as to integrate by parts, differentiating the left term and integrating the right. Because the right term integrates to zero, the boundary terms in the integration by parts formula cancel, and we are left with an integral of the same form as (\ref{entry_integral}) on which we can repeat the process. Thus we obtain a family of expressions
	\[ M_{jk} = \frac{(-1)^n}{2\pi} \int_{0}^{2\beta\pi} h_n(x) e^{i(kv(x)-jx)} dx,\ n \leq r, \]
	with each $h_n$ being $(r-n)$-times differentiable and defined by the recurrence relation
	\[ h_0 = h,\ h_{n+1} = -i \left[ \frac{h_n}{j-k v'} \right]'. \]
	We find by induction that 
	\[ h_n = i^n \sum_{l=0}^{n} \frac{k^{l} w_{n, l}(x)}{(j-kv'(x))^{n+l}}, \]
	with $w_{n,l}$ having the recurrence relation
	\begin{align*}
		w_{n,l} &= w'_{n-1,l} + (n+l-1) v'' w_{n-1,l-1},&& 0<l<n,\\
		w_{n,0} &= w_{n-1,0}',&& n>0, \\
		w_{n,n} &= 2n v'' w_{n-1,n-1}, && n>0,\\
		w_{0,0} &= h.
	\end{align*}

	By induction, we see that each $w_{n,l}$ has the form
	\[ w_{n,l} = \sum_{l'=0}^{n-l} \omega_{n,l,l'}(v'',\ldots,v^{(n-l+2)}) h^{(l')},\]
	where $\omega_{n,l,l'}$ are degree $l$ homogeneous polynomials with positive coefficients. (The $\omega_{n,0,l'}$ are constants as a result, and thus issues of existence of derivatives do not arise.) 
	
	Setting
	\[ W_{n,l} = \sup_{x\in[0,2\pi]} \frac{|w_{n,l}(x)|}{|h(x)|} \leq \sum_{l'=0}^{n-l} \omega_{n,l,l'}(\Upsilon_1,\ldots, \Upsilon_{n-l+1}) H_{l'}, \]
	we have
	\[ |M_{jk}| \leq \frac{1}{2\pi}\int_0^{2\beta\pi} 
	\sum_{l=0}^n \frac{ W_{n,l}|h(x)| |k|^l}{|j-k v'(x)|^{n+l}} dx,\]
	from which (\ref{entry_differentiable}) follows by H\" older's inequality.
	
	For (c), we use the $2\beta\pi$-periodicity of the integrand of (\ref{entry_integral}) to move the contour of integration. When $j/k < p_2$, we shift the contour of integration by $-i \zeta\sgn k $ in the complex plane so
	\begin{equation} M_{jk} = \frac{1}{2\pi}\int_{0}^{2\beta\pi} h(x-i\zeta\sgn k ) e^{ikv(x-i\zeta\sgn k ) - ij(x-i\zeta\sgn k )} dx. \label{shifted_entry_integral}\end{equation}
	
	We now use our bounds on derivatives of $h$ and $v$ to bound elements of this expression, beginning with the argument of the exponential. 
	
	Applying Taylor's theorem to $\Im v(x+i\xi)$, we have
	\[ \Im v(x - i\zeta\sgn k) =  -  \zeta \sgn k v'(x) - \frac{1}{2}\zeta^2 \Im v''(\xi) \]
	 for some $\xi \in  \Lambda^\beta_\delta$. This gives us that
	\begin{align*} \Re (ikv(x - i\zeta\sgn k)) &\leq \zeta k \sgn k |v'(x)| + |k| \frac{1}{2} \zeta^2 \Upsilon_{1,\delta} \\
	&\leq  \zeta |k| \left(\mu_1 + \frac{\zeta\Upsilon_{1,\delta}}{2}\right) \\ &
	\leq \zeta |k| p_1, \end{align*}
	where the last inequality results from the definition of $\zeta$ in the statement of the lemma.
	
	We can bound $h(x-i\zeta\sgn k)$ by using that the Lipschitz constant of $\log h$ on $ \Lambda^\beta_\delta$ is $\sup |h'/h| \leq H_{1,\delta}$. As a result,
	\[ |h(x-i\zeta\sgn k)| \leq |h(x)| e^{|i\zeta\sgn k| H_1} = |h(x)| e^{\zeta H_{1,\delta}}.\]
	
	Thus when we take absolute values on (\ref{shifted_entry_integral}) we obtain that
		 \[ |M_{jk}| \leq \frac{1}{2\pi} \int_0^{2\beta\pi} |h(x)| e^{\zeta H_{1,\delta}} e^{\zeta \sgn k (k p_1 - j) } dx. \]
	Using that $\sgn k = \sgn (j - p_1 k)$ for $j > p_1 k$ and H\"older's inequality yields (\ref{entry_analytic}).
	
		The proof of (c) for $j/k < p_2$ is analogous, with the contour shifted in the opposite direction.

\qed \end{proof}

Given Lemma \ref{l:abstract_entry}, Theorem \ref{t:entry_p} is an elementary result. It is necessary only to check that the conditions for the theorem imply the conditions for the lemma, and vice versa for the results. 

\begin{proof}[
	Theorem \ref{t:entry_p}]

From (\ref{transfer_action}), the transfer operator $\tro$ of a map $f \in \bupd$ has action
\[ \tro \phi(x) = \sum_{n=1}^{b} \sigma v'(x+2b\pi) \phi(v(x+2b\pi)), \]
where $\sigma = \sgn v'(0)$. (Note that $v$ is monotonic and so $\sigma v' = |v'|$). 

Since $\lambda_2^{-1}<|v'|<\lambda_1^{-1}$, we can apply Lemma \ref{l:abstract_entry} with $h = \sigma v'$.

Suppose that $f$ satisfies (\ref{distortion_diff_frac}). Then we can set $\Upsilon_n = C_n$ for all $n \leq r$, as the definition of $C_n$ in (\ref{distortion_diff_frac}) and of $\Upsilon_n$ in Lemma \ref{l:abstract_entry} are the same. We can also set
\[ |v^{n+1}| \leq \left|\frac{v^{(n+1)}}{v'}\right| |v'| \leq \frac{C_n}{\min\{|\lambda_1|,|\lambda_2|\}} = H_n < \infty. \] 

This gives us what we need for Lemma \ref{l:abstract_entry}(b), and so there exist $W_{r,n}$ such that
\[ L_{jk} \leq \frac{\|v'\|_1}{2\pi} \sum_{n=0}^{r} \frac{W_{r,n} |k|^n}{|j - \lambda_m^{-1} k |^{n+r}} \]
where $\lambda_m$ is $\lambda_1$ for $j/k > \lambda_1^{-1}$ and $\lambda_2$ for $j/k < \lambda_2^{-1}$.

We can eliminate the sum by using that for $j/k > p_1$, 
\[ \frac{|k|^r}{|j - \lambda_m^{-1} k |^r} = \frac{1}{|j/k - \lambda_1^{-1}|^r} \leq \frac{1}{(p_1 - \lambda_1^{-1})^r}, \]
and similarly for $p_2$.
Furthermore, since $v'$ does not change sign, 
$\|v'\|_1 = |v(2\pi\beta) - v(0)| = 2\pi$.

Thus there exists a constant $C$ depending on the distortion constants $C_r$, expansion bounds $\lambda_{1,2}$ and constants $p_{1,2}$ such that for $j/k \notin [p_2,p_1]$ or $k = 0$,
\[ L_{jk} \leq \frac{C}{|j - \lambda_m^{-1} k|}\leq \frac{C}{|j - p_m^{-1} k|}, \] 
which implies the bound for maps in (\ref{distortion_diff_frac}) from Theorem \ref{t:entry_p}.

Similarly, suppose that $f$ satisfies (\ref{distortion_diff_frac}). Then $\Upsilon_{1,\delta} = C_{1,\delta} <\infty$,
and
\[ \sup_{v\in \Lambda^\beta_\delta} \left|v'\right| \leq e^{\delta C_{1,\delta}}\cdot\sup_{x\in[0,2\beta\pi)} |v'(x)| <\infty, \]
and hence by Lemma \ref{l:abstract_entry}(c) there exists $C>0$ and $\zeta \in (0,\delta]$ such that for $j/k > p_1$,
$ |L_{jk}| < C e^{-\zeta |j - p_2 k| }$,
and similarly for $j/k < p_2$.
\qed \end{proof}

Theorem \ref{t:entry_np} also follows from Lemma \ref{l:abstract_entry}, since we can piggyback off the relation between Chebyshev polynomials and Fourier series:
\[ T_k(\cos\theta) = \frac{1}{2} e^{ik\theta} + \frac{1}{2} e^{-ik\theta}.\]
However, because the cosine function on $[0,2\pi)$ is two-to-one with critical points at $0$ and $\pi$, the proof is less straightforward than for Theorem \ref{t:entry_p}. In particular, we will have to address how to turn the transfer operator of a map in $\bunp$ into the sum of operators of the form (\ref{calm_defn}), with regard to the two-to-one nature of the transformation. We will then need to examine how distortion bounds translate quantitatively under this transformation. Once we have done these, the bounds follow easily.

\begin{proof}[
	Theorem \ref{t:entry_np}]
From the definition of transfer operators (\ref{transfer_action}) and the orthogonality relation for the Chebyshev basis, we obtain the following formula for Chebyshev basis matrix elements of transfer operators of maps in $\bunp$:
\[ L_{jk} = \frac{t_j}{\pi} \sum_{\iota\in I} \int_{-1}^1 \frac{\sigma_\iota}{\sqrt{1-x^2}} v_\iota'(x) T_k(v_\iota(x))) T_j(x)\ dx, \]
where $\sigma_\iota = \sgn v_\iota'$, $t_j = 2-\delta_{j0}$, and the sum is taken over the branches of the map. Under the transformation $x = \cos\theta$ and using that $T_k(x) = \cos(k\cos^{-1}x)$, we find that $L_{jk}$ is related to a Fourier basis matrix entry for a weighted transfer operator:
\begin{equation} L_{jk} = \frac{t_j}{\pi} \sum_{\iota\in I} \sigma_\iota \int_{0}^\pi v_\iota'(\cos \theta) \cos(k \cos^{-1} v_\iota(\cos\theta))) \cos j\theta \ dx. \label{t-entry-np-1}\end{equation}
Based on this, we set $h_\iota = v_\iota'\circ\cos$ for each $\iota\in I$. These functions $h_\iota$ are $2\pi$-periodic.

Defining $\upsilon_{\iota+}:=\cos^{-1}\circ v_\iota\circ \cos$ and $\upsilon_{\iota-} = 2\pi - \upsilon_{\iota+}$, we find
\begin{align*}   L_{jk} &= \frac{t_j}{\pi} \sum_{\iota\in I} \sigma_\iota \int_{0}^\pi h_\iota(\theta) \cos(k \upsilon_{\iota+}(\theta)) \cos j\theta \ dx\\
&= \frac{t_j}{4\pi} \sum_{\iota\in I} \sigma_\iota \int_{0}^{\pi} h_\iota(\theta) \sum_{\pm}  \left(e^{i(k\upsilon_{\iota\pm}(\theta) - j\theta)} + e^{i(k\upsilon_{\iota\pm}(\theta)+j\theta)}\right) d\theta.
\end{align*}

Continuing $\upsilon_{\iota\pm}$ differentiably to the interval $[0,2\pi]$ and using that the integrands are symmetric about $\pi$, we can finally rewrite the transfer operator in the form 
\begin{equation}L_{jk} = \frac{t_j}{8 \pi} \sum_{\iota\in I,\,\pm} \sigma_\iota \int_{0}^{2\pi}  h_\iota(\theta) \left(e^{i(k\upsilon_{\iota\pm}(\theta) - j\theta)} + e^{i(k\upsilon_{\iota\pm}(\theta)+j\theta)}\right) d\theta. \label{cheby_transfer_exp}\end{equation}

If neither or both of $v_\iota(-1)$ and $v_\iota(1)$ are a singular point of the $\cos^{-1}$ transformation (i.e. $-1$ or $1$), then the $\upsilon_{\iota\pm}$ are differentiably defined on the circle $\R/2\pi\Z$. If one of these values is, then $\upsilon_{\iota+}$ will continue across the critical points on either side to $\upsilon_{\iota-}$ and so their concatenation  $\upsilon_\iota$ is a differentiable map on $\R/4\pi\Z$. Thus, if we define the sets $I_c = \left\{ \iota\in I: |v_\iota(\{ \pm 1 \})\cap \{\pm 1\}| = 1 \right\}$ 
and
$ I' = \left(I\backslash I_c \times \{+,-\}\right) \cup I_c, $
and set $\beta_{\iota'} = 1 + \mathbf{1}_{I_c}(\iota')$,
we have

\begin{equation}L_{jk} = \frac{t_j}{8\pi} 
\sum_{\iota\in I'} \sigma_\iota \int_{0}^{2\pi} \sum_{b=0}^{\beta_{\iota'}-1}  h_{\iota'}(\theta+2\pi b) e^{ik\upsilon_{\iota'} (\theta+ 2\pi b)} \left(e^{- ij\theta} + e^{ij\theta}\right) d\theta. \label{cheby_transfer_exp2}\end{equation}

Clearly, the summands are two-element sums of Fourier coefficient matrix elements of operators of the form \ref{calm_defn}. The following lemma, whose proof is for ease of exposition in Appendix \ref{a:distortion-proofs}, shows that that the relevant bounds on $\upsilon_{\iota'}$ and $h_\iota$ hold uniformly for all $\iota$:
\begin{lemma}\label{l:cheby_distortion}
	Suppose $f \in \bunp$ with partition spacing constant $\Xi$ and $I'$ is defined as above. Then
	\begin{enumerate}[(a)]
		\item If the $v_\iota$ satisfy $($\ref{distortion_diff_frac}$)$ with the same distortion constants $C_n,\ n\leq r$, then for $n\leq r$ there exist $\Upsilon_n, H_n <\infty$ depending only on $C_m, m\leq n$ and $\Xi$ such that
		\[ \sup_{\theta \in [0,2\pi], \iota'\in I'} \upsilon_{\iota'}^{(n+1)}(\theta) \leq \Upsilon_{n}\]
and
 \[ \sup_{\theta \in [0,2\pi], \iota'\in I'} \left| \frac{h_{\iota'}^{(n)}(\theta)}{h_{\iota'}(\theta)} \right| \leq H_n < \infty. \]
		\item If the $v_\iota$ obey $($\ref{distortion_ana}$)$ with the same distortion constant $C_{1,\delta}$, then there exists $\zeta \in (0,\delta]$, $\Upsilon_{1,\zeta}, H_{1,\zeta} <\infty$ depending only on $\zeta$, $C_{1,\delta}$ and partition spacing constant $\Xi$ such that
					\[ \sup_{\theta \in \Lambda^{\beta_{\iota'}}_\zeta, \iota'\in I'} \upsilon''_{\iota'}(\theta) \leq \Upsilon_{1,\zeta}\]
					and
					\[ \sup_{\theta \in \Lambda^{\beta_{\iota'}}_\zeta, \iota'\in I', \pm} \left| \frac{h'_{\iota'}(\theta)}{h_{\iota'}(\theta)} \right| \leq H_{1,\zeta} < \infty. \]
	\end{enumerate}
\end{lemma}

 Setting $\tilde \mu = [-\check{\lambda}^{-1},\check{\lambda}^{-1}]$ and $\tilde p = [-p,p]$, Lemma \ref{l:cheby_distortion} means we can apply Lemma \ref{l:abstract_entry} to each summand in (\ref{cheby_transfer_exp2}). Up to a constant factor $G$ to be discussed later we have 
that if $f$ satisfies (\ref{distortion_diff_frac}) then there exist $W_{r,n}$ such that for $j>pk\geq 0$,
\[ |L_{jk}| \leq 2\frac{t_j}{8\pi} G \sum_{n=0}^{r} \frac{W_{r,n} k^n}{(j - \check\lambda^{-1} k )^{n+r}}.\]
Similarly, if $f$ satisfies (\ref{distortion_ana}) there exists $\zeta' \in (0,\zeta]$ such that for $j>pk\geq 0$,
\[ |L_{jk}| \leq 2\frac{t_j}{8\pi} G e^{\zeta' (H_{1,\zeta} - (j - p k))} \]
which gives the decay rates stated in Theorem \ref{t:entry_np} by the same means as in the proof of Theorem \ref{t:entry_p}.

However, we need to check that the constant factor
\[ G = \sum_{\iota' \in I'} \int_0^{2\beta_{\iota'}\pi} |v'_\iota(\cos\theta)| d\theta \]
is in fact finite. 

We convert back to a sum over $I$ by collapsing the sum over $\pm$ for $\iota\in I\backslash I_c$, obtaining 
\[ G = 2 \sum_{\iota\in I} \int_0^{2\pi} |v'_\iota(\cos\theta)|d\theta. \]
We then make the two-to-one change of variable $x = \cos\theta$ to find that 
\begin{align*} G &= 4\sum_{\iota\in I} \int_{-1}^{1} |v'_\iota(x)|  \frac{1}{\sqrt{1-x^2}} dx 
\\ &\leq 4 \sum_{\iota\in I} \int_{-1}^{1} \frac{(1+2C_1) |\O_\iota|}{2\sqrt{1-x^2}} dx = 4\pi (1+2C_1) < \infty,\end{align*}
%
where the first inequality is a result of Lemma \ref{l:unp-consistency}(b).

This concludes the proof of Theorem \ref{t:entry_np}.
\qed \end{proof}

\begin{remark}\label{r:entry_unif}
	Elements of Fourier and Chebyshev transfer operator matrices are uniformly bounded, with 
	\[ |L_{jk}| \leq 1 \]
	for maps in $\bupd$ and
	\[ |L_{jk}| \leq (2-\delta_{j0}) (2+ 4C_1) 
	\]
	for maps in $\bunp$.
	
	This follows by applying Lemma \ref{l:abstract_entry}(a) in the proofs of Theorems \ref{t:entry_p}-\ref{t:entry_np}.
\end{remark}

\begin{remark}
	The uniform C-expansion condition (\ref{aue_definition}) \Rii{is the natural expansion condition for any choice of spectral basis on an interval}. Our reasoning is as follows. If one wishes to use oscillatory integral techniques on these basis functions as in Lemma \ref{l:abstract_entry}, it is best for the wavelength of the basis functions to be approximately spatially constant. 
	However, wavelengths of sufficiently high-order spectral basis functions on intervals will always be much smaller towards the endpoints. Potential theory \cite{Trefethen13} tells us that the optimal transformation to even out high-order basis functions across the interval is always the cosine transformation.
	\end{remark}

We now turn to proving the main theorem, Theorem \ref{t:main}, and its corollaries. Our idea is to perturb $\tro$ such that the associated coefficient matrix is block-upper-triangular (in the Fourier case, with the ordering of basis elements $e_0, e_{1}, e_{-1}, e_2,\ldots$). This isolates the top block $E_N \to E_N$, which then approximates the corresponding $E_N \times E_N$ block of the full, unperturbed transfer operator, yielding convergence on domain $E_N$.

We summarise this using the following lemma, \Rii{where we do not require our Banach space $E$ to be $BV$}.

\begin{lemma}\label{l:strong-convergence}
	\Rii{Let $E$ be a Banach space such that $\E_N := (\id - \prj_N) \tro \prj_N$ is an endomorphism on $E$.}
	
	Suppose $\tro$ has a spectral gap \Rii{on $E$}. Then
	\begin{equation} \| \tro_N - \tro |_{\spc_N} \|_\spc = \| \E_N\|_\spc \label{strong-convergence-transfer} \end{equation}
	and
	\begin{equation} \| \sch_N - \sch |_{\spc_N}  \|_\spc \leq \frac{\|\sch\|_\spc \| \E_N \|_\spc}{1 - \|\sch\|_\spc \| \E_N \|_\spc}. \label{strong-convergence-solution}\end{equation}
\end{lemma}

\begin{proof}
	The first equality (\ref{strong-convergence-transfer}) arises simply because 
	$\tro_N - \tro |_{\spc_N}  = (\id-\prj_N)\tro_N |_{\spc_N}  = \E_N$.
	
	Let 
	\[\tilde{\tro}_N := \tro - \E_N = \tro_N + \tro (\id - \prj_N).\] 
	Recalling that we defined $\ti$ to be the Lebesgue integral functional and $u$ an element of $\spc_N$ with $\ti u = 1$, let us also define
	$ \tilde{\sch}_N = (\id - \tilde{\tro}_N + u \ti)^{-1}$.
	If $\|\E_N\|$ is small enough, this is well-defined, since $\tilde{\sch}_N =  (\id + \sch \E_N)^{-1} \sch$ and thus 
	\begin{equation} \|\tilde{\sch}_N - \sch\| \leq 	\frac{\|\sch\| \|\E_N\|}{1-\|\sch\| \|\E_N\|}.\label{strong-convergence-solution-bound}\end{equation}
	
	For $\phi \in \spc_N$, we have that 
	\[ \tilde{\sch}_N^{-1} \phi = \phi - \tilde{\tro}_N \phi + u \ti \phi = \phi - \tro_N \phi + u \ti \phi \in \spc_N, \]
	and thus $\tilde{\sch}_N^{-1} |_{\spc_N}$ is an endomorphism on $\spc_N$, is equal to $\sch_N^{-1}$. Consequently, $\sch_N |_{\spc_N} = \tilde{\sch}_N |_{\spc_N}$, which combined with (\ref{strong-convergence-solution-bound}) yields as desired (\ref{strong-convergence-solution}).
\qed\end{proof}

The following lemma is then required to connect $\|\E_N\|_{BV}$ to spectral matrix coefficients.
	\begin{lemma}\label{l:bvnorm}
		Suppose $\mathcal{F}: BV(\Lambda)\to BV(\Lambda)$ is an operator for $\Lambda$ either $[0,2\pi)$ or $[-1,1]$. Let the matrix $D= (k \delta_{jk})_{j,k\in\Z}$ and $\check D = (k\delta_{(j-1)k})_{j,k\in\N}$. 
		
		If $\mathcal{F}$ has Fourier coefficient matrix $F$, then
		\begin{equation} \|\mathcal{F}\|_{BV} \leq 2\pi (\left\| D F \|_{\ell^2} + \|F\|_{\ell^2}\right).\label{l-bvnorm-four}\end{equation}
		
		Similarly, if $\mathcal{F}$ has Chebyshev coefficient matrix $F$, then
		\begin{equation} \|\mathcal{F}\|_{BV} \leq 2\pi (\left\| \check C \check D F \check C^{-1} \|_{\ell^2} + \|\check C F \check C^{-1}\|_{\ell^2}\right),\label{l-bvnorm-cheb}\end{equation}
		where $\check C = (t_k^{-1/2}\delta_{jk})_{j,k\in\N}$.
	\end{lemma}
\begin{proof}
	Consider first the Fourier case. Then since $\frac{1}{\sqrt{2\pi}}\|\cdot\|_{L_2} \leq \|\cdot\|_{BV} \leq \sqrt{2\pi} \|\cdot\|_{H^1}$, 
	\begin{equation} \|\mathcal{F}\|_{BV} \leq 2\pi\|\mathcal{F}\|_{L_2 \to H^1} = 2\pi\left( \|\mathcal{D}\mathcal{F}\|_{L_2} + \|\mathcal{F}\|_{L_2}\right).  \label{l-bvnorm-1} \end{equation}
	By the Plancherel equality, $\|\mathcal{D}\mathcal{F}\|_{L_2} = \|DF\|_{\ell^2}$ and $\|\mathcal{F}\|_{L_2} = \|F\|_{\ell^2}$. This gives the required bound in (\ref{l-bvnorm-four}).
	
	Consider instead the Chebyshev case. Define the Jacobi weight function $\mathfrak{j}(x) = \sqrt{1-x^2}$, and the Sobolev spaces $\check H^k \subset L_2([-1,1],1/\mathfrak{j})$, $k\geq0$ with norm
	\begin{equation} \|\phi\|_{\check H^k} = \sum_{n=0}^k \int_{-1}^1 \mathfrak{j}^{2n-1} |\phi^{(n)}|^2 dx. \label{cheby-sobolev-spaces}\end{equation}
	Note that $\check H^0=L^2([-1,1],1/\mathfrak{j})$.

	If $G$ is the set of even functions on $\R/2\pi\Z$, simple trigonometric manipulations show that the operator $\mathcal{C}: \phi \mapsto \frac{1}{2} \phi\circ\cos$ is an isometry from $\check H^k$ to $G \cap H^k([0,2\pi))$ and similarly from $BV([-1,1])$ to $G\cap BV([0,2\pi))$. Thus,
	\begin{align*} \|\mathcal{F}\|_{BV([-1,1])} &= \|\mathcal{CFC}^{-1}\|_{G\cap BV([0,2\pi))} \\&
	\leq 2\pi\left( \|\mathcal{D}\mathcal{CFC}^{-1}\|_{G\cap L_2} + \|\mathcal{CFC}^{-1}\|_{G \cap L_2}\right),   \end{align*}
	where the inequality comes from (\ref{l-bvnorm-1}). We can then convert back to $\check H^0$ to get the inequality
	\[ \|\mathcal{F}\|_{BV([-1,1])} \leq 2\pi \left( \|\mathcal{C}^{-1}\mathcal{DC}\mathcal{F}\|_{\check H^0} + \|\mathcal{F}\|_{\check H^0}\right) .\]
	
	We can convert these operator norms into matrix norms as follows. The Chebyshev polynomial basis is an orthogonal basis for $\check H^0$ with $\|T_k\|_{\check H^0} = \sqrt{\pi/t_k}$ and furthermore 
	the functions
	\[ \mathcal{C}^{-1}\mathcal{DC} T_k = k\sin(k\cos^{-1}x) \]
	are orthogonal in $\check H^0$ with norms $k \sqrt{\pi/t_k}$ respectively. The resulting Plancherel equality results in (\ref{l-bvnorm-cheb}).
\qed\end{proof}

We now have the requisite results to tie together to prove Theorem \ref{t:main}.

\begin{proof}[
	Theorem \ref{t:main}]
	Maps in $\upd$ have a spectral gap in $BV$ as they are uniformly expanding with bounded distortion. Since maps in $\unp$ have a forward iterate that is uniformly expanding with bounded distortion by Theorem \ref{t:conjugacy} in Appendix \ref{a:expanding-conjugacy}, they also have a spectral gap in $BV$.
	
	Suppose $E_N$ is the Fourier coefficient matrix of $\E_N$ and the expansion coefficient of the associated map $f$ is $\lambda > 1$. Then given $p \in (\lambda,1)$, there exists an appropriate spectral decay function $K$ such that when $|j|\geq |k|$,
	\[ |L_{jk}| \leq K(|j|-p|k|). \]
	
	Now suppose $f$ satisfies (\ref{distortion_diff_frac}) for some $r \geq 2$. Then $K(M) = C M^{-r}$ for some $C > 0$, and so
	\begin{align*} \|E_N\|_{\ell^2}^2 &\leq  \sum_{k=-N}^N\sum_{j=N+1}^{\infty} \left(|L_{jk}|^2 + |L_{-jk}|^2\right) \\
		&\leq \sum_{k=-N}^N \sum_{j=N+1}^{\infty} 2 C^2 (j-p|k|)^{-2r} \\
		& \leq \frac{2C^2}{2r-1}  \sum_{k=-N}^N (N - p|k|)^{1-2r},
	\end{align*}
	by converting to an integral. We can then take the supremum of the summands to obtain
	\[ \sum_{k=-N}^N (N - p|k|)^{1-2r} \leq (2N+1) (N - pN)^{1-2r} \leq \frac{3}{(1-p)^{2r-1}} N^{2-2r}\]
	and thus 
	\[ \|E_N\|_{\ell^2}^2 \leq \frac{6}{(2r-1)(1-p)^{2r-1}} N^2 K(N)^2. \]
	Similarly,
	\[ \|D E_N\|_{\ell^2}^2 \leq \frac{6}{(2r-2)(1-p)^{2r-2}} N^3 K(N)^2, \]
	where $D$ is as in Lemma \ref{l:bvnorm}.
	
	Hence as a result of Lemma \ref{l:bvnorm}, there exists a function $K'\in \kappa(\DD_r)$ such that $\|\E_N\| \leq N\sqrt{N} K'(N)$.
	
	Suppose $f$ instead satisfies (\ref{distortion_ana}). Then for some $\zeta \in (0,\delta]$ there exists $p>1$ such that for all $|j|\geq|k|$, $L_{jk} \leq C e^{-\zeta(|j|-p|k|)}$. Consequently,
	\[ \|E_N\|_{\ell^2}^2 \leq \sum_{k=-N}^N\sum_{j=N+1}^{\infty}2e^{-2\zeta(j-p|k|)} \leq \frac{4N}{\zeta^2} e^{-2\zeta(1-p)N} \]
	with a comparable result for $DE_N$. Thus, there exists a function $K'\in \kappa(\DA_\delta)$ such that
	\[ \|\E_N\| \leq N K'(N) \leq N\sqrt{N} K'(N). \]
	
	Similarly, we get the same results up to constants in the Chebyshev case: the $C$ matrices are unproblematic as $\|C\|_{\ell^2} = 1$ and $\|C^{-1}\|_{\ell^2} = 2$. 
	
	We therefore have by Lemma \ref{l:strong-convergence} that if $f$ satisfies some distortion condition $(\D)$ then there exists $K' \in \kappa(\D)$ such that for any $N$ and $\phi$ in $\tro_N$,
	\[ \|\tro_N \phi - \tro \phi \|_{BV} \leq N\sqrt{N} K'(N) \| \phi\|_{BV} \]
	and if $N$ is sufficiently large,
	\[ \|\sch_N \phi - \sch \phi\|_{BV} \leq \frac{\|\sch\|_{BV} N\sqrt{N} K'(N)}{1 - \|\sch\|_{BV} N \sqrt{N} K'(N)} \leq 2N\sqrt{N} \|\sch\|_{BV} K'(N) \|\phi\|_{BV}, \]
	which is what was required for Theorem \ref{t:main}.
	\qed\end{proof}
	
	Corollary \ref{c:acim} is a direct result of this convergence and Theorem \ref{t:solution}:
	\begin{proof}[
		Corollary \ref{c:acim}]
		We know from Theorem \ref{t:solution} that $\ac = \sch u$. We have also defined $\ac_N = \sch_N u$, recalling that $u$ lies in $E_N$. As a result, by Theorem \ref{t:main},
		\[ \| \ac_N - \ac \|_{BV} = \| \sch_N u - \sch u \|_{BV} \leq N\sqrt{N} \bar{K}(N) \|u\|_{BV}, \]
		as required.
	\qed\end{proof}
	Note that here, unlike in Theorem \ref{t:main}, we actually have that estimates converge in norm to the true values.
	
	Corollary \ref{c:sum} also follows directly from Theorems \ref{t:solution} and \ref{t:main}.
	\begin{proof}[
		Corollary \ref{c:sum}]
		We know from Theorem \ref{t:solution} that on $\zsp$, the space of zero integral functions, the solution operator $\sch$ is identical to $\sum_{n=0}^{\infty} \tro^n$. We then need only apply the second part of Theorem \ref{t:main} to get the required inequality.
	\qed\end{proof}	
	

\begin{remark}\label{r:optimal-zeta}
	When a map satisfies $($\ref{distortion_ana}$)$, one might be interested in the best rate of decay one can get for $\|\E_N\|_{BV}$, which controls the convergence of estimates. In the non-periodic case one can show that 
	\begin{equation} \lim_{N\to\infty} \frac{1}{N} \log\|\E_N\|_{BV} = \sup_{\iota\in I, x \in [0,2\pi]} |\Im \upsilon_\iota(x+i\zeta)| - \zeta. \label{analytic-error}\end{equation}
	The value of $z$ where the supremum in (\ref{analytic-error}) is maximised will have $\Im \upsilon_\iota'(z) = 0$; if this value of $z$ varies continuously with $\zeta$, then it will have a maximum when $|\Re \upsilon_\iota'(z)|$ is $1$ or $-1$.
	Thus, one expects the right-hand side of (\ref{analytic-error}) to be maximised for
			\begin{equation} \zeta = \min\left\{\inf \left\{ |\Im z| \mid z \in (\upsilon_\iota')^{-1}(\{\pm 1\}), \iota\in I \right\},\delta\right\}. 
			\label{zeta_optimal} \end{equation}
		
			The result is the same in the periodic case but with $v_\iota$ substituted for $\upsilon_\iota$.)
\end{remark}

\section*{Acknowledgements}

The author wishes to thank Georg Gottwald for his support, advice and comments on the manuscript. The author would also like to thank Sheehan Olver for discussions on spectral methods and ApproxFun, and Ian Melbourne, Viviane Baladi, Maria Jose Pacifico, Peter Koltai for their advice, references and illuminating discussions.

 The author was supported by a Research Training Program scholarship, and by the Erwin Schr\"odinger Institute, Vienna during a visit in May 2016.

\begin{appendix}
	\section{The relationship between uniform expansion and uniform C-expansion}\label{a:expanding-conjugacy}

%
In this paper we have stipulated that maps on non-periodic domains satisfy a so-called uniform C-expansion condition rather than the usual uniform expansion condition. Neither of these conditions imply the other: in fact it is not hard to construct non-pathological examples of uniformly-expanding maps which are not uniformly C-expanding (see Figure \ref{f:expanding-conjugacy-1}).

\begin{figure}[htbp]
	\centering
	\includegraphics[width=80mm
	]{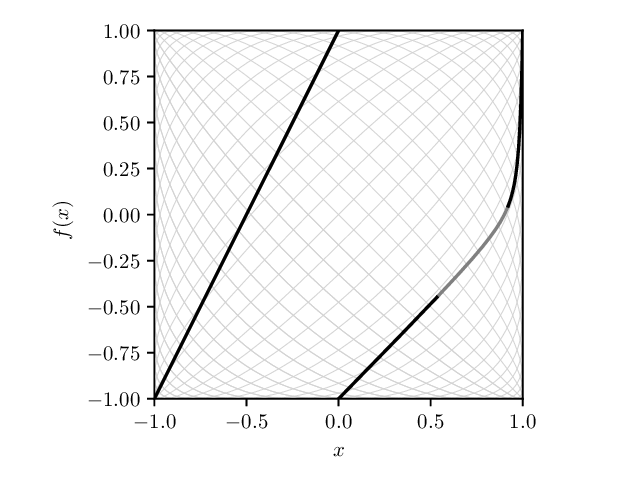}
	\caption{In black, an example of a map $f\in\uunp$ ($\lambda = 0.98^{-1}$) which is not uniformly C-expanding ($\check\lambda \approx 0.763$). Non-C-expanding parts of $f$ marked in mid grey. In light grey, lines of unit C-expansion (i.e. curves $\psi(x)$ for which $(\cos^{-1}\circ\psi\circ\cos)' = \pm 1$).}
	\label{f:expanding-conjugacy-1}
\end{figure}

However in practice, maps in $\uunp$ are generally also in $\unp$. For example, all piecewise linear maps in $\uunp$ lie in $\unp$. (In particular, if $f$ is the $k$-tupling map, the uniform C-expansion parameter for $f$ is $\check{\lambda} = \sqrt{k}$.) 
A map in $\uunp$ typically fails to be in $\unp$ if its graph ``aims'' towards the walls of the domain. To preserve the Markov structure, such a heading must be facilitated by a kink in the map. This is illustrated in Figure \ref{f:expanding-conjugacy-1}. 

However, if we now consider only maps that are Markov with bounded distortion, we find close connections between C-expansion and classical expansion. In fact, a positive lower bound on one implies a positive lower bound on the other, which may be seen by an adaptation of the proof of Theorem \ref{t:conjugacy} below.
%

Importantly, uniformly C-expanding maps eventually become uniformly expanding under iteration and vice versa, according to the following theorem.
\begin{theorem}\label{t:conjugacy}
	Suppose $f \in \uunp$ (resp. $f \in \unp$).
%
	Then there exists $n_* \in \N$ such that $f^n \in \unp$ (resp. $\uunp$) for all $n\geq n_*$. Each $f^n$ satisfies the same distortion conditions as $f$, with possibly different constants.
\end{theorem}

\Rii{

\begin{remark}
	Since iterates of a map have an exponentially growing number of branches, for computational purposes it may be more effective simply to compute a conjugacy of a map which is C-expanding.

	It is in fact possible to construct, for a map $f \in \uunp$ (resp. $\unp$), an analytic diffeomorphism $\eta_f$ such that $f_c = \eta_f\circ f \circ \eta_f^{-1} \in \unp$  (resp. $\uunp$).
	
	
	Furthermore, if $f$ satisfies $($\ref{distortion_diff_frac}$)$ then so will $f_c$, and if $f$ satisfies $($\ref{distortion_ana}$)$ for some $\delta > 0$ then there exists $\delta' > 0$ for which $f_c$ satisfies $(\DA_{\delta'})$.
\end{remark}
}

Thus, maps in $\uunp$ and in $\unp$ have the same dynamical properties and can additionally be converted from one class to the other. 
We emphasise that the crucial assumption here is bounded distortion. 
%
\\

We now prove the results stated above, beginning with Theorem \ref{t:conjugacy}.

\begin{proof}[
	Theorem \ref{t:conjugacy}]
	
	Suppose $f \in \uunp$ with $|f'|>\lambda$ and distortion constant $C_1$. Then $f^n \in \uunp$ with $|(f^n)'| > \lambda^n$ and distortion constant bounded by $C_1\frac{1-\lambda^{-n-1}}{1-\lambda^{-1}}$ \cite{Gouezel04}. Let us use the notation $f^n = g$ with branches $\mathcal{P}_\iota, \iota \in I^n$. 
	
	Suppose $x \in \mathcal{P}_{\iota}$ for some $\iota \in I^n$. Then 
		\[ \frac{1-|x|}{|\sgn(xg'(x)) - g(x)|} \geq \frac{v_{\iota_x}(\sgn(xg'(x))) - x}{|g(v_{\iota_x}(\sgn(xg'(x)))) - g(x)|}, \]
		and by the intermediate value theorem there exists $w \in \mathcal{P}_\iota$ such that 
		\[ \frac{v_{\iota_x}(\sgn(xg'(x))) - x}{|g(v_{\iota_x}(\sgn(xg'(x)))) - g(x)|} = \frac{1}{|g'(w)|}.\]
		Using Lemma \ref{l:unp-consistency}(a) we find that 
		\[ \frac{1}{|g'(w)|} > \frac{e^{-2C_1}}{|g'(x)|},\]
		and so for all $x \in \cup_{\iota\in I} \O_\iota$,
		\[\sqrt{\frac{1-x^2}{1-(g(x))^2}}|g'(x)| \geq \sqrt{\frac{1+|x|}{|\sgn(xg'(x)) + g(x)| }e^{-2C_1} |g'(x)|} \geq \sqrt{\frac{1}{2} e^{-2C_1}\lambda^n} > 1\] 
	for $n$ sufficiently large.
	
	The map $f^n$ is full-branch Markov with bounded distortion, by the above satisfies (\ref{aue_definition}), and by Lemma \ref{l:ppc-composition} satisfies (\ref{ppc}). Consequently, $f^n \in \unp$.
	\\
	Now, suppose $f \in \unp$ and let $n \in \N^+$. 
	For the remainder of this proof, we will use subscript notation for forward iterates: $x_n = f^n(x)$. We will additionally call two points $x, y \in \Lambda$ {\it $n$-companions} if there exists a sequence $(\iota_{j})_{j=1,\ldots, n}$ such that $x_{j-1}, y_{j-1} \in \O_{\iota_j}$ for $j\leq n$.
	
	 Given $x \in \Lambda$, choose $y, z$ such that $x$, $y$ and $z$ are $n$-companions, $y_n = -1$ and $z_n = 1$. Then by the mean value theorem, there exists $w$ between $y$ and $z$ such that
	\begin{equation} |(f^n)'(w)| = \frac{|z_n-y_n|}{|z -y|} \geq \frac{2}{\pi} \check{\lambda}^n. \label{t-conj-fnw} \end{equation}
	
	Now, since $w$ lies between $y$ and $z$, it is an $n$-companion of $x$, $y$ and $z$. We will therefore relate $|(f^n)'(w)|$ to $|(f^n)'(x)|$ using bounded distortion.
	
	We expand their quotient out using the chain rule and rewrite:
	\begin{align} \frac{|(f^n)'(x)|}{|(f^n)'(w)|} &= \prod_{j=1}^{n} \frac{|f'(x_{j-1})|}{|f'(w_{j-1})|} \notag \\
		&= \prod_{j=1}^{n} \frac{|v'_{\iota_j}(x_j)^{-1}|}{|v'_{\iota_j}(w_j)^{-1}|} \notag \\
		&=  e^{\sum_{j=1}^{n}\left(\log |v'_{\iota_j}(w_j)| - \log |v'_{\iota_j}(x_j)|\right)} \notag\\
		&\geq e^{-\sum_{j=1}^{n}\left|\log |v'_{\iota_j}(w_j)| - \log |v'_{\iota_j}(x_j)|\right|}.\label{t-conj-vsum} \end{align}
	
	 We then bound the summands using (\ref{distortion_basic}) and the fact that $v''_\iota/v'_\iota = (\log |v'_\iota|)'$:
	\[ \left|\log |v'_{\iota_j}(w_j)| - \log |v'_{\iota_j}(x_j)|\right| \leq C_1 |w_j - x_j| \leq C_1 \check{\lambda}^{j-n} \pi.\]
	 The sum in (\ref{t-conj-vsum}) can thus be collapsed to give
	\[ \frac{|(f^n)'(x)|}{|(f^n)'(w)|} \geq e^{-\sum_{j=1}^n C_1 \check{\lambda}^{j-n} \pi} > e^{-C_1 \pi (1-\check\lambda)^{-1}}.\] 
	
	Combining this with (\ref{t-conj-fnw}) gives us that
	\[ |(f^n)'(x)| \geq e^{-C_1 \pi (1-\check\lambda)^{-1}} \frac{2}{\pi} \check{\lambda}^n, \]
	which implies that $f^n$ is uniformly expanding for sufficiently large $n$. Since as before $f^n$ satisfies all the non-expansion conditions to be in $\uunp$, we have that $f^n \in \uunp$.
\qed\end{proof}
	\section{Results on conditions (\ref{distortion_basic}) and (\ref{ppc})}\label{a:distortion}


In this appendix, we 
prove some properties possessed by maps in $\unp$ used through the rest of the paper. We first
give some ``non-local'' properties of the bounded distortion condition (\ref{distortion_basic}), and then prove that (\ref{ppc}) is preserved under iteration.
%
%

We first prove a lemma relating bounded distortion constants to bounds on derivatives of the map. The properties summarised in Lemma \ref{l:unp-consistency} are mostly standard, but we improve the upper bound in part (b) from the exponentially large $e^{2C_1}$ to a computationally more useful $1+2C_1$.

\begin{lemma}\label{l:unp-consistency}
	Suppose $f: [-1,1] \to [-1,1]$ is full-branch Markov with bounded distortion. Suppose the distortion constant of $f$ is $C_1$. Then for all $\iota\in I$: 
	\begin{enumerate}[(a)]
	\item For all $x,w \in [-1,1]$,
	\[ e^{-2C_1} \leq \frac{|v_\iota'(x)|}{|v_\iota'(w)|} \leq e^{2C_1}; \]
	\item For all $x \in [-1,1]$,
	\[ e^{-2C_1} \frac{|\O_\iota|}{2} \leq |v_\iota'(x)| \leq (1+2C_1) \frac{|\O_\iota|}{2}. \]
	\end{enumerate}
\end{lemma}

\begin{proof}[Proof of Lemma \ref{l:unp-consistency}]
	Part (a) is a standard result \cite{Gouezel04, Korepanov16}.
	
	To prove (b), we have that as a result of the intermediate value theorem there exists some $w \in [-1,1]$ such that
	\[ v_\iota'(w) = \frac{v_\iota(1)-v_\iota(-1)}{2} = \frac{|\O_\iota|}{2}.\]
	
	By part (a), $ e^{-2C_1} \leq \frac{|v_\iota'(x)|}{|v_\iota'(w)|}$.
	Additionally, the fundamental theorem of calculus gives that
	\[ v_\iota'(x) = v_\iota'(w) + \int_{w}^x v_\iota''(\xi) d\xi,\]
	and consequently 
	\begin{align*} |v_\iota'(x)| &\leq |v_\iota'(w)| + \int_{x}^w C_1 |v_\iota'(\xi)| |d\xi|\\
		&\leq \frac{|\O_\iota|}{2} + C_1 \int_{-1}^1 |v_\iota'(\xi)| d\xi
		= \left(\frac{1}{2}+C_1\right) |\O_\iota|, \end{align*}
	as required.

\end{proof}

\begin{remark}\label{r:unp-consistency-ana}
	Similarly, suppose that a map $f \in \bunp$ obeys analytic distortion condition (\ref{distortion_ana}) with  constant $C_{1,\delta}$. Then for all $x,w \in \check\Lambda_\delta$,
		\[ e^{-2C_{1,\delta}\cosh \delta } \leq \frac{|v_\iota'(x)|}{|v_\iota'(w)|} \leq e^{2C_{1,\delta}\cosh \delta }.\] 
\end{remark}

We now prove that the partition spacing condition (\ref{ppc}) is preserved under composition. Consequently, $\unp$ and $\bunp$ are closed under composition. 

\begin{lemma}\label{l:ppc-composition}
	
	Suppose $f$ and $g$ are Markov maps on $[-1,1]$ satisfying (\ref{ppc}), and that in addition $f$ has bounded distortion with parameter $C\brf_1$ and $g$ has uniform expansion parameter $\lambda\brg > 0$.
	
	Then $g\circ f$ satisfies (\ref{ppc}). 
\end{lemma}

\begin{proof}
	Let $\O_{\phi\gamma}= v\brf_\phi \left( v\brg_\gamma (\Lambda)\right)$ be a branch set of $g\circ f$. Let $\Pa \in \partial \Lambda$, i.e. $\Pa = \pm 1$. 
	
	Since $\O_{\phi\gamma} = v\brf_\phi\left(\O\brg_\gamma\right)$, by Lemma \ref{l:unp-consistency}(c) we have
	\begin{equation*} \frac{\left|\O_{\phi\gamma}\right|}{\left|\O\brg_\gamma\right|} \leq \left(1+2C\brf_1\right) \frac{\left|\O\brf_\phi\right|}{2},\end{equation*}
	and thus a preliminary bound on our ratio of interest:
	\begin{equation} 
	\frac{\left|\O_{\phi\gamma}\right|}{d(\O_{\phi\gamma},\Pa)} \leq \frac{\left(1+2C\brf_1\right) \frac{1}{2} \left|\O\brg_\gamma\right| \left|\O\brf_\phi\right|}{d(\O_{\phi\gamma},\Pa)}.\label{t-a-ppc-comp1}  \end{equation}
	
	We are interested in intervals for which $\Pa \notin \O_{\phi\gamma}$. If $\Pa\in \O_{\phi\gamma}$, then we need $\Pa \in \O\brf_\phi$ and $\hat f_\phi(\Pa)=:\tau \in \O\brg_\gamma$. Note that since $\Pa \in \partial \O\brf_\phi$, then $\tau\in\partial\Lambda$. Therefore, intervals $\O_{\phi\gamma}$ which do not contain $\Pa$ either have $\Pa \notin \O\brf_\phi$ or $\tau \notin \O\brf_\gamma$.
	
	We split into cases accordingly. In the first case where $\Pa \notin \O\brf_\phi$, we have that since $\O\brg_\gamma = v\brg_\gamma([-1,1])$, its length must be less than $2/\lambda\brg$. Since $\O_{\phi\gamma} \subseteq \O\brf_\phi$, we must have 
	$d\left(\O_{\phi\gamma},\Pa\right)\geq d\left(\O\brf_\phi,\Pa\right)$.
	Therefore from (\ref{t-a-ppc-comp1}),
	\[ \frac{\left|\O_{\phi\gamma}\right|}{d(\O_{\phi\gamma},\Pa)} \leq (1+2C\brf_1)\frac{1}{\lambda\brg} \Xi\brf,\] 
	where we used that $|\O\brg_\gamma|<|\Lambda|/\lambda\brg$ from the expansion assumption.
	
	For the second case, let $\Pb \in \partial \Lambda$ be such that $v\brf_\phi(\Pb)$ lies in between $\O_{\phi\gamma}$ and $\Pa$ and let $\Pc \in \partial \O\brf_\gamma$ such that  $v\brf_\phi(\Pc)$ is the nearest point in $\O_{\phi\gamma}$ to $\Pa$ (and thus, $\Pb$). Then $d(\O_{\phi\gamma},\Pa)$ is the length of the interval $[v\brf_\phi(\Pc),\Pa]$, which is bigger than the length of the interval $[v\brf_\phi(\Pc),v\brf_\phi(\Pb)] = v\brf_\phi([\Pc,\Pb])$. Using Lemma \ref{l:unp-consistency} the length of this last interval can be bounded:
	\[ \frac{\left|v\brf_\phi([\Pc,\Pb])\right|}{\left|[\Pc,\Pb]\right|} = 
	 \frac{\int_{\Pc}^{\Pb} \left| (v\brf_\phi)'(x)\right| |dx|}{\left|[\Pc,\Pb]\right|} 
	 \geq e^{-2C\brf_1} \frac{1}{2} \left|\O\brf_{\phi}\right|.  \]
	Furthermore, the distance between $\Pc$ and $\Pb$ is precisely $d(\O\brg_\gamma,\Pb)$. 
	
	Combining these results with (\ref{t-a-ppc-comp1}), we find that
	\[ \frac{\left|\O_{\phi\gamma}\right|}{d(\O_{\phi\gamma},\Pa)} \leq\left(1+2C\brf_1\right) e^{2C\brf_1} \frac{\O\brg_\gamma}{d(\O\brg,\Pb)} \leq \left(1+2C\brf_1\right) e^{2C\brf_1} \Xi\brg. \]
	
	Combining the two cases, then, we find that 
	\[  \Xi^{(g\circ f)} \leq \left(1+2C\brf_1\right)\max\left\{e^{2C\brf_1} \Xi\brg, \frac{1}{\lambda\brg} \Xi\brf\right\}, \]
	as required.
\end{proof}

	\section{Proof of Lemma \ref{l:cheby_distortion}}\label{a:distortion-proofs}

In this appendix we will prove 
Lemma \ref{l:cheby_distortion}, which states that standard properties of $f$ (e.g. differentiability of the distortion) imply the properties of $\cos^{-1}\circ f \circ \cos$ required to apply Lemma \ref{l:abstract_entry} in the proof of Theorem \ref{t:main}.

We remark that while the partition position condition (\ref{ppc}) is crucial for the proof in general, it is not necessary if one restricts to maps that only satisfy ($\DD_1$).

We also emphasise that this lemma gives very loose bounds for the $\Upsilon_n$ and $\Upsilon_{1,\delta}$, and that in practice one is best served by calculating these constants directly from the $\upsilon_\iota$.

We will first state and prove two lemmas upon which Lemma \ref{l:cheby_distortion} relies, and then prove the latter. 

\begin{lemma}\label{l:cheby_distortion_s}
	Suppose the map $f \in \bunp$ is piecewise $C^{n+1}$, choose $\sigma \in \partial\Lambda = \{\pm 1\}$ and let $\tau = v_\iota(\sigma)$. Define the gradient of the chord
	\[ \hat S_{\iota,\sigma}(x) = \frac{\tau - v_\iota(x)}{\sigma -x}.\]

	Then 
	\[ \hat S_{\iota,\sigma}^{(n)}(w) = \frac{1}{n+1} v^{(n+1)}(w) \]
	for some $w$ directly between $\sigma$ and $x$.
\end{lemma}
\begin{proof}
	We can show by induction that for all $n \geq 0$
	\[ \hat S_{\iota,\sigma}^{(n)}(x) = n! \frac{\tau - \sum_{m=0}^n \frac{1}{m!} v_\iota^{(m)}(x)(\sigma - x)^m}{(\sigma-x)^{n+1}}.\]
	Since $v_\iota(\sigma) = \tau$, the lemma follows by Taylor's theorem. 
\qed\end{proof}

\begin{lemma}\label{l:cheby_distortion_t}
	Suppose $f \in \bunp$, $\iota\in I$ and $\sigma \in \partial\Lambda= \{\pm 1\}$ such that $v_\iota(\sigma) \notin \partial\Lambda$. Let $\tau_{\iota,\sigma} = \sigma \sgn v_\iota'(0)$ and
	 \[T_{\iota,\sigma}(x) = 1 -v_\iota(x)/\tau_{\iota,\sigma}.\]
	
	Then for $x\in [-1,1]$,
	$ T_{\iota,\sigma}(x) \geq \Xi^{-1} |\O_\iota|.$
	If $f$ satisfies analytic distortion condition (\ref{distortion_ana}) then there exists $\zeta \in (0,\delta]$ and $\mathfrak{K}_\zeta > 0$ such that for $z \in \check{\Lambda}_\zeta$
	$ |T_{\iota,\sigma}(z)| \geq \mathfrak{K}_\zeta |\O_\iota|.$
\end{lemma}
\begin{proof}
	Recalling that $\tau_{\iota,\sigma}^2 = 1$ we can write
	\[ T_{\iota,\sigma}(x) = \tau_{\iota,\sigma} (\tau_{\iota,\sigma} - v_\iota(\sigma)) + \tau_{\iota,\sigma} (v_\iota(\sigma)-v_\iota(x)).\]
	Since $\tau_{\iota,\sigma} - v_\iota(\sigma)$ has the same sign as $\tau_{\iota,\sigma}$, the first term can be written as a positive quantity $|\tau_{\iota,\sigma} - v_\iota(\sigma)|$ which is equal to $d(\tau_{\iota,\sigma},\O_\iota)$. By the partition spacing condition (\ref{ppc}), we have $d(\tau_{\iota,\sigma},\O_\iota) \geq \Xi^{-1} |\O_\iota|$.
	
	Furthermore, one can apply Taylor's theorem to the second term to get
	that $ \tau_{\iota,\sigma} (v_\iota(\sigma)-v_\iota(x)) = \tau_{\iota,\sigma} (\sigma - x) v_\iota'(w)$
	for some $w$ between $x$ and $\sigma$. Since $v_\iota'$ keeps its sign on $[-1,1]$ and the sign of $\sigma - x$ is simply the sign of $\sigma$, the definition of $\tau_{\iota,\sigma}$ means that the second term is positive on $[-1,1]$. Thus, for $x \in [-1,1]$,
	$ T_{\iota,\sigma}(x) \geq \Xi^{-1} |\O_\iota|.$
	
	On the analytic domain $\check{\Lambda}_\zeta$ the situation is more complicated. We write that
\begin{equation}
		\Re T_{\iota,\sigma}(x) = d(\tau_{\iota,\sigma},\O_\iota) - \tau_{\iota,\sigma} \Re \left(v_\iota(x)-v_\iota(\Re x)\right) - \tau_{\iota,\sigma} \left(v_\iota(\Re x) - v_\iota(\sigma)\right) 
		,\label{l-cheby_distortion_t_re}
\end{equation}
and bound terms from below.

Set $\mathfrak{C}_\zeta = e^{2C_{1,\delta}\cosh\zeta}$. We have that for any point $x$ 
in $\check{\Lambda}_{\zeta}$, 
$ |v'(x)| \leq \mathfrak{C}_\zeta \frac{1}{2}|\O_\iota|$
as a result of Remark \ref{r:unp-consistency-ana}. We will use this fact in the following discussion.

The Bernstein ellipse $\check{\Lambda}_{\zeta}$ has major axis $\cosh\zeta\cdot[-1,1]$ and minor axis $i\sinh\zeta\cdot[-1,1]$. As a consequence every point $w$ in $\check{\Lambda}_\zeta$ has $\Re w \leq \cosh\zeta$ and $|\Im w| \leq \sinh\zeta$. 

We have by Taylor's theorem that
\[ \tau_{\iota,\sigma} \Re \left(v_\iota(x)-v_\iota(\Re x)\right) = \tau_{\iota,\sigma} \Re \left(v_\iota'(\Re x) i\Im x - \frac{v_\iota''(w)}{2} \Im x^2\right)\]
for $w$ between $x$ and $\Re x$ (i.e. in $\check{\Lambda}_\zeta$). Thus,
\[ |\tau_{\iota,\sigma} \Re \left(v_\iota(x)-v_\iota(\Re x)\right)| \leq \frac{C_{1,\delta} \mathfrak{C}_\zeta |\O_\iota|}{4} \sinh^2\zeta.\]

Furthermore,
\[ \tau_{\iota,\sigma} \left(v_\iota(\Re x) - v_\iota(\sigma)\right) = \sigma^{-1} \sgn v_\iota'(0) (\Re x - \sigma) v_\iota'(w)\]
for $w$ between $\Re x$ and $\sigma$, i.e. in $\check{\Lambda}_\zeta \cap \R$. Since $v_\iota' \neq 0$ on $\check{\Lambda}_\zeta$ because of the bounded distortion condition (\ref{distortion_ana}), and $v_\iota'$ must be real on $\check{\Lambda}_\zeta \cap \R$ as it is real on $[-1,1]$ and analytic on the whole interval, we have $\sgn v_\iota'(w) = \sgn v_\iota'(0)$ and so
\begin{align*} \tau_{\iota,\sigma} \left(v_\iota(\Re x) - v_\iota(\sigma)\right) &= (\Re x/\sigma - 1) |v_\iota'(w)|\\ &\leq (\cosh\zeta - 1)|v_\iota'(w)| \\&\leq (\cosh\zeta - 1) \mathfrak{C}_\zeta \frac{|\O_\iota|}{2}.\end{align*}

As a result we have from (\ref{l-cheby_distortion_t_re})
\[ |T_{\iota,\sigma}(x)|\geq \Re T_{\iota,\sigma}(x) \geq \left(\Xi^{-1} - \frac{\mathfrak{C}_\zeta}{4}\left(C_{1,\delta} \sinh^2\zeta + 2\cosh\zeta-2\right)\right)|\O_\iota|.\]
When $\zeta$ is small enough, the term multiplying $|\O_\iota|$ is positive. 
\qed\end{proof}

With these lemmas in hand, we can now prove Lemma \ref{l:cheby_distortion}.

\begin{proof}[
	Lemma \ref{l:cheby_distortion}]
	
	We begin with the first part of part (a), bounding derivatives of the $\upsilon_\iota$. We will do this by first proving a formula for the derivatives of $\upsilon_\iota$ and then bounding terms in this formula to get overall bounds.
	
	Let $\pi_n := n \mod 2$.
	 We claim that
	\begin{equation} \upsilon_\iota^{(n+1)}(\cos^{-1} x) = \sum_{q+r+s \leq n} a_{q,r,s,n}(x) Y_{\iota,1}^{q,n}(x) Y_{\iota,-1}^{r,n}(x) v^{(s+1)}(x),\label{l-cheby-distortion-claim}\end{equation}
	where $a_{q,r,s,n}$ are polynomials in $x$ with coefficients independent of $f$, and
	\begin{equation}  Y_{\iota,\sigma}^{m,n}(x) = \begin{cases}
		(1- x\sigma^{-1})^{\frac{\pi_{n}}{2}} \left( S_{\iota,\sigma}^{-1/2} \right)^{(m)},& v_\iota(\sigma) \in \{-1,1\}, \\
	(1- x\sigma^{-1})^{\frac{\pi_{n}}{2}} \left( T_{\iota,\sigma}^{-1/2} \right)^{(m)},& v_\iota(\sigma) \notin \{-1,1\}.
	\end{cases} \label{l-cheby-distortion-y}
	\end{equation}
	
	We prove this claim by induction. Suppose without loss of generality that $\sgn v_\iota' = 1$. 
	
	When $n = 0$, we have that 
	\[ \upsilon_\iota'(\cos^{-1} x) = \sqrt{\frac{1-x}{1-v_\iota(x)}} \sqrt{\frac{1+x}{1+v_\iota(x)}} v_\iota'(x).\]
	From (\ref{l-cheby-distortion-y}), we find that 
	\[ Y_{\iota,\sigma}^{0,0}(x) =
	\sqrt{\frac{1-x\sigma^{-1}}{1-v_\iota(x)\sigma^{-1}}}, \]
	and thus (\ref{l-cheby-distortion-claim}) follows for $n=0$.
	
	Suppose, then, that (\ref{l-cheby-distortion-claim}) is true for some $n$. Then
	\[  \upsilon_\iota^{(n+2)}(\cos^{-1}(x)) = \sqrt{1-x^2} (\upsilon_\iota^{(n+1)}\circ \cos^{-1})'(x).\]
	All we need to show is that $\sqrt{1-x\sigma^{-1}} Y_{\iota,\sigma}^{m,n}(x)$ and $\sqrt{1-x\sigma^{-1}} (Y_{\iota,\sigma}^{m,n})'(x)$ can be written as a product of $Y_{\iota,\sigma}^{m,n+1}(x)$ (and for the derivative possibly also $Y_{\iota,\sigma}^{m+1,n+1}(x)$), and polynomials in $x$. In the case where $v_\iota(\sigma) \in \{-1,1\}$, we have
	\begin{align*} \sqrt{1-x\sigma^{-1}} Y_{\iota,\sigma}^{m,n}(x) &= (1- x\sigma^{-1})^{\frac{\pi_{n}+1}{2}} \left( S_{\iota,\sigma}^{-1/2} \right)^{(m)}\\ &= (1-x\sigma^{-1})^{\pi_{n}} Y_{\iota,\sigma}^{m,n+1}(x)\end{align*}
	and
	\begin{align*} \sqrt{1-x\sigma^{-1}} (Y_{\iota,\sigma}^{m,n})'(x) &= (1- x\sigma^{-1})^{\frac{\pi_{n}+1}{2}} \left( 		S_{\iota,\sigma}^{-1/2} \right)^{(m+1)} \\&\qquad
			 - \pi_n \sigma^{-1} (1- x\sigma^{-1})^{\frac{\pi_{n}-1}{2}} \left( S_{\iota,\sigma}^{-1/2} \right)^{(m+1)} \\
		&= (1-x\sigma^{-1})^{\pi_{n}} Y_{\iota,\sigma}^{m+1,n+1}(x) \\&\qquad
			- \pi_n \sigma^{-1} (1-x\sigma^{-1})^{\pi_{n}-1} Y_{\iota,\sigma}^{m,n+1}(x)\\
		&= (1-x\sigma^{-1})^{\pi_{n}} Y_{\iota,\sigma}^{m+1,n+1}(x) - \pi_n \sigma^{-1} Y_{\iota,\sigma}^{m,n+1}(x), \end{align*}
	where in the last line we removed the $(1-x\sigma^{-1})^{\pi_{n}-1}$ element from the last term by using that the last term is zero unless $\pi_n=1$. The relation when $v_\iota(\sigma)\notin\{-1,1\}$ is clearly analogous, from which the claim falls.
	
	We now attempt to bound the expression in (\ref{l-cheby-distortion-claim}). To bound the $Y_{\iota,\sigma}^{m,n}$, we need to bound derivatives of $S_{\iota,\sigma}^{-1/2}$ and $T_{\iota,\sigma}^{-1/2}$. One may show by induction that for $n \geq 1$ there exist multivariate polynomials $q_n$ such that 
	for any function $U$,
	\begin{equation} (U^{-1/2})^{(n)} = U^{-1/2} q_n\left(\frac{U'}{U}, 
	\ldots,\frac{U^{(n)}}{U}\right).\label{l-cheby-distortion-derivs}\end{equation}
	
	By Lemma \ref{l:cheby_distortion_s}, we have that when $v_\iota(\sigma) \in \{-1,1\}$
	\[ |{S_{\iota,\sigma}^{(n)}(x)}| = \frac{1}{n+1} |v^{(n+1)}(w)|\]
	for some $w \in [-1,1]$.
	Using distortion bound ($\DD_{n}$) and Lemma \ref{l:unp-consistency} we can bound this again to get that 
	\[ |{S_{\iota,\sigma}^{(n)}(x)}| \leq \frac{C_n e^{2C_1}}{n+1} |v'(x)|.\]
	We also have that
	$ |S_{\iota,\sigma}(x)| = |v'(w)| > e^{-2C_1} |v'(x)|$
		for some $w \in [-1,1]$.
		
	Substituting these bounds into (\ref{l-cheby-distortion-derivs}) we find that
	\[ \left| \left( S_{\iota,\sigma}^{-1/2} \right)^{(n)}\right| \leq e^{C_1} |v'|^{-1/2} |q_n|\left(\frac{C_1 e^{4C_1}}{2},\ldots,\frac{C_n e^{4C_1}}{n+1}\right) \]
	when $v_\iota(\sigma)\in\{-1,1\}$. 
	
	Similarly, we have that %
	$ |T_{\iota,\sigma}^{(n)}(x)| = |v^{(n)}(x)| \leq C_n |v'(x)| $
	and, by Lemma \ref{l:cheby_distortion_t}, when $v_\iota(\sigma) \notin \{-1,1\}$ that
	$ |T_{\iota,\sigma}(x)| \geq \Xi^{-1} |\O_\iota| \geq 2\Xi^{-1} e^{-2C_1} |v'(x)|. $
	These bounds can be substituted into (\ref{l-cheby-distortion-derivs}) similarly to give
		\[ \left| \left( T_{\iota,\sigma}^{-1/2} \right)^{(n)}\right| \leq \sqrt{\frac{\Xi}{2}} e^{C_1} |v'|^{-1/2} |q_n|\left(\frac{C_1 \Xi e^{2C_1}}{2},\ldots,\frac{C_n \Xi e^{2C_1}}{2}\right) \]
		when $v_\iota(\sigma)\notin\{-1,1\}$.

	Thus, there exist constants $\mathfrak{k}^{m,n}$ depending on the distortion constants and partition spacing constant such that for all $\iota\in I$ and $\sigma\in\{-1,1\}$, we have
	$ |Y_{\iota,\sigma}^{m,n}(x)|\leq \mathfrak{k}^{m,n} |v'(x)|^{-1/2}$.
	
	Returning to (\ref{l-cheby-distortion-claim}), we have that since $|v^{(s+1)}(x)|\leq C_s |v'(x)|$,
	\[ |\upsilon_\iota^{(n+1)}(\cos^{-1} x)| \leq \sum_{q+r+s \leq n} |a_{q,r,s,n}|(1)\ \mathfrak{k}^{q,n}\mathfrak{k}^{r,n} C_s,\]
	for $x \in [-1,1]$, and thus $|\upsilon_\iota^{(n+1)}(\theta)|$ is bounded by the same constant for $\theta \in [0,2\pi\beta_{\iota'})$.
	
	The proof of the first part of part (b) is essentially the same as the above with $n=1$. The major difference is that we apply Remark \ref{r:unp-consistency-ana} and the second bound in Lemma \ref{l:cheby_distortion_t} instead of Lemma \ref{l:unp-consistency} and the first bound, respectively. We also use that $\cos^{-1}\check{\Lambda}_{\zeta} = \Lambda^{\beta_{\iota'}}_\zeta$ so bounds on $\upsilon^{(n)}(\theta)$ transfer directly to bounds on $\upsilon^{(n)}(\cos^{-1}(x))$.
	
	The second parts of (a) and (b) are much more straightforward. In both cases we seek to bound 
	\begin{equation} \left|\frac{h_\iota^{(n)}}{h_\iota}\right| = \frac{\left|(v_\iota' \circ \cos)^{(n)}\right|}{\left|v_\iota' \circ \cos\right|} \label{l-cheby-distortion-h}\end{equation}
			on appropriate domains. 
			The $n$th derivative of $v_\iota'\circ\cos$ can be written as a linear combination of $v_\iota^{(m+1)}\circ\cos, m \leq n$ with coefficients of trignometric polynomials. Trigonometric polynomials are bounded on $[0,2\pi\beta_{\iota'}]$ and $\Lambda^{\beta_{\iota'}}_\zeta$; on these respective domains, the $|v_\iota^{(m+1)}\circ\cos|$ are bounded by $C_m |v_\iota'\circ \cos|$ and by $C_{1,\delta} |v_\iota'\circ \cos|$ for $m=1$ respectively. Thus, we find that (\ref{l-cheby-distortion-h}) are bounded by constants depending on $C_{m}$, $m\leq n-1$, and in the analytic case on $\zeta$ (which parameterised $\check{\Lambda}_\zeta$) and $C_{1,\zeta}$.
\qed\end{proof}

%
%
		\section{Explicit bounds on the norm of the solution operator in $BV$}\label{a:bv-norm}
	
	In \cite{Korepanov16}, explicit a priori bounds on decay of correlations were stated in the Lipschitz norm. Specifically, if a map on $[0,1]$ has expansion coefficient $\lambda$ and (\ref{distortion_basic}) distortion constant $C_1$, then with $\zsp$ the space of zero-integral functions on $[0,1]$, the following bound holds:
	\begin{align}
	R &= \frac{2C_1}{1-\lambda^{-1}}\notag\\
	D &= 4e^{R}(1+R),\notag\\
	\xi &= \frac{1}{2} e^{-R} (1-\lambda^{-1}),\notag\\
	\| \tro^n |_{\zsp} \|_{\Lip} &\leq  D e^{-\xi n}.\label{korepanov-decay}
	\end{align}
	
	In this appendix we sketch how these explicit bounds work through to bound $\|\sch\|_{BV}$.
	
	Let $\Lip([0,1])$ be the space of Lipschitz functions on the interval $[0,1]$ with the usual norm.
	
	Suppose that $\|\tro^{n}|_{\zsp}\|_{\Lip} \leq K_n < 1/2$. 	Suppose that $g \in BV([0,1]) \cap \zsp $ with $\|g\|_{BV} = 1$. Let $\hat g_{\mathfrak{n}}$ be the piecewise linear interpolant to $g$ at the points $0,\frac{1}{\mathfrak{n}},\frac{2}{\mathfrak{n}},\ldots,1$. It can be seen that $\Lip \hat g_{\mathfrak{n}} \leq \mathfrak{n}$ and $ \|\hat g_{\mathfrak{n}} - g \|_1 \leq \frac{1}{2\mathfrak{n}}$.
	
	Consequently,
	\begin{align*}\|\tro^n g\|_1 &\leq \|\tro^n \hat g_{\mathfrak{n}}\|_1 + \|\tro^n (g - \hat g_{\mathfrak{n}})\|_1 \\
	&\leq \frac{1}{5} \|\tro^n \hat g_{\mathfrak{n}}\|_{\Lip} + \| g - \hat g_{\mathfrak{n}}\|_1	\\	
	&\leq \frac{K_n}{5} \left(\Lip \hat g_{\mathfrak{n}} + \|\hat g_{\mathfrak{n}}\|_{\infty}\right) + \| g - \hat g_{\mathfrak{n}}\|_1 \\
	& \leq \frac{K_n}{5} \left(\mathfrak{n} + 1\right) + \frac{1}{2\mathfrak{n}},
	\end{align*}
	where we used that $\|h\|_{\Lip} \geq 5\|h\|_1$ and $\|h\|_{BV} \geq \|h\|_{\infty}$ for $h\in\zsp$. 
	
	Setting $\mathfrak{n}=\lceil K_n^{-1/2} \rceil$, we have 
	\[ \|\tro^n g\|_1 \leq \frac{\sqrt{K_n}(7+4\sqrt{K_n})}{10} \leq \sqrt{K_n}. \]
	
	Hence, as a result of the standard $BV$ Lasota-Yorke inequality \cite{Galatolo14} we find that
	\begin{equation} \| \tro^{m+n} g \|_{BV} \leq \frac{5}{4} |\tro^{m+n} g|_{BV} \leq  \frac{5}{4}(\lambda^{-m}  C_1 \sqrt{K_n}). \label{a-bv-norm-1} \end{equation}
	
	Using that $\|\tro^n|_\zsp\|_{\Lip}\leq D e^{-\xi n}$ from (\ref{korepanov-decay}), and choosing 
	\begin{align*} n &= \left\lceil \frac{4+2\log(\max\{C_1,1\} \sqrt{D}))}{\xi}\right\rceil\\
	m &= \left\lceil \frac{2}{\log\lambda}\right\rceil,\end{align*}
	we have 
	\[\| \tro^n |_{\zsp} \| \leq \frac{e^{-4}}{ \min\{1,C_1^{-2}\} D^{-1}}  =: K_n < 1/2.\]
	Consequently from (\ref{a-bv-norm-1}) we have that $ \| \tro^{m+n} \|_{BV} \leq \frac{5}{2}e^{-2} \leq \frac{2}{5}$.
	
	As a result,
	\begin{equation} \left\| \sum_{k=0}^\infty \tro^k \right\|_{BV} \leq  \left\| \sum_{k=0}^\infty \tro^{(m+n)k} \right\|_{BV} \left\| \sum_{k=0}^{m+n-1} \tro^k \right\|_{BV}  \leq  \frac{5}{3}  (m+n)C', \label{transfer_sum_bound} \end{equation}
	where 
	$C' := 1 + \frac{1}{3}\frac{C_1}{1-\lambda^{-1}} \geq \sup_{n\in \N} \| \tro^n\|_{BV} \leq $ This bounding property of $C'$ is the result of the Lasota-Yorke inequality and the fact that $\|g\|_{BV}\geq 3\|g\|_{1}$ for $g\in BV\cap\zsp$.
	
	As a result of (\ref{solution-sum}), we finally obtain the a priori bound on the solution operator
	\begin{equation}\|\sch\|_{BV} \leq 1 + \frac{5}{3}  (m+n)C' (3+C'). \label{solution_bound} \end{equation}

		%
		%
		%
		%
		
\end{appendix}
	
	\bibliographystyle{siam}
	\bibliography{Spectral-bib}
\end{document}